\documentclass{article}

\usepackage{theorem}
\usepackage{enumitem,multicol}
\usepackage{theorem}
\usepackage{amsmath,amssymb,amsfonts,graphicx,float,bm}
\usepackage{yhmath}

\def\BBox{\rule{2mm}{3mm}}
\def\QED{\hfill$\BBox$}
\newenvironment{proof}
{\begin{rm}\par\smallskip\noindent{\bf Proof.}\quad}{\QED\end{rm}}

\theorembodyfont{\rmfamily}
\newtheorem{thm}{Theorem}[section]
\newtheorem{remark}[thm]{\bfseries Remark}    %

\newtheorem{prop}[thm]{\bfseries Proposition} %

\newtheorem{conj}[thm]{\bfseries Conjecture}

\newtheorem{defn}[thm]{\bfseries Definition}
\newtheorem{exmp}[thm]{\bfseries Example}
\newtheorem{prob}[thm]{\bfseries Problem}

\begin{document}

\title{A two-dimensional topological representation theorem for matroid polytopes of rank 4}

\author{
Hiroyuki Miyata
\\
Department of Computer Science, \\
Gunma University, Japan\\
{\tt hmiyata@gunma-u.ac.jp}
}

\maketitle

\begin{abstract}
The Folkman-Lawrence topological representation theorem, which states that every (loop-free) oriented matroid of rank $r$ can be represented as a pseudosphere arrangement on the $(r-1)$-dimensional sphere $S^{r-1}$, is one of the most outstanding results in oriented matroid theory.
In this paper, we provide a lower-dimensional version of the topological representation theorem for uniform matroid polytopes of rank $4$.
We introduce $2$-weak configurations of points and pseudocircles ($2$-weak PPC configurations) on $S^2$ and prove that 
every uniform matroid polytope of rank $4$ can be represented by a $2$-weak PPC configuration. 
As an application, we provide a proof of Las Vergnas conjecture on simplicial topes for the case of uniform matroid polytopes of rank $4$.
\end{abstract}

\maketitle

\section{Introduction}
Oriented matroid theory, introduced by Bland and Las Vergnas~\cite{BL78}, and Folkman and Lawrence~\cite{FL78}, is a rich theory, which has connections with various fields such as discrete geometry, topology and algebraic geometry, and operations research (see \cite{BLSWZ99}).
One of the driving forces of the theory is the Folkman-Lawrence topological representation theorem~\cite{FL78}, which claims that there is a one-to-one correspondence
between (loop-free) oriented matroids of rank $r$ and equivalence classes of pseudosphere arrangements on the $(r-1)$-dimensional sphere $S^{r-1}$.
This theorem enables one to translate combinatorial problems into topological problems, and vice versa.

When $r=3$, the topological representations of oriented matroids given by the theorem can be transformed into {\it pseudoline arrangements} in the projective plane $\mathbb{P}^2$.
Pseudoline arrangements are a topological generalization of line arrangements in $\mathbb{P}^2$ and are a classical subject of discrete geometry  (see \cite{G72}).
One of the famous results on pseudoline arrangements is Ringel's homotopy theorem~\cite{R56}, which states that every pair of simple arrangements of $n$ pseudolines can be transformed into each other by a finite sequence of triangle flips. 
This theorem immediately leads to the fact that every pair of uniform oriented matroids of rank $3$ on the same ground set can be transformed into each other by a finite sequence of mutations.
Also, it is well-known that any pseudoline arrangement contains a triangle region (see \cite{G72}), which leads to that any simple oriented matroid of rank $3$ has an acyclic reorientation with exactly three extreme points. 

When $r \geq 4$, the topological representations of oriented matroids become more complicated. Pseudosphere arrangements on $S^{r-1}$ ($r \geq 4$) are usually very difficult to visualize, and there remain many open problems on rank $r (\geq 4)$ oriented matroids.
One of the outstanding problems is Cordovil-Las Vergnas conjecture (mentioned in \cite{RS88}), which claims that every pair of uniform oriented matroids on the same ground set can be transformed into each other by a finite sequence of mutations.
In terms of pseudosphere arrangements, this conjecture claims the existence of a higher-dimensional analogue of Ringel's homotopy theorem.
The conjecture is known to be true for realizable oriented matroids~\cite{RS88}, but in general case, the conjecture is wide open.
To be able to apply mutation to an oriented matroid,  there must exist a simplicial tope  (in terms of pseudosphere arrangements, a simplicial cell) in it.
The existence of a simplicial tope in every simple oriented matroid was conjectured by Las Vergnas~\cite{LV80}, but the answer is still unknown except for oriented matroids of rank 3~\cite{G72},  realizable oriented matroids~\cite{S79}, and oriented matroids admitting extensions in general position that are Euclidean~\cite{EM82}.

To obtain a better understanding of the case $r \geq 4$, one of the natural approaches might be to represent oriented matroids as lower-dimensional geometric objects.
Such an approach is quite common in discrete geometry.
To list a few examples in polytope theory, 3-polytopes are often represented as 3-connected plane graphs (Steinitz' theorem), and 4-polytopes are often visualized as Schlegel diagrams in $\mathbb{R}^3$.
Gale diagrams are also a common tool to visualize high-dimensional polytopes with few vertices in low dimensional spaces (see \cite{Z95}).
It is also common to study high-dimensional geometric objects through suitably chosen projections.
However, there is not so much work on oriented matroids in this direction. 
There are remarkable results that improve the topological representations of oriented matroids into the topological representations by PL-sphere arrangements~\cite{EM82} and that
propose another type of the topological representation theorem by pseudoconfigurations of points for rank $3$ oriented matroids~\cite{GP84},
but the dimensions of the topological representations are the same as the original.
If we restrict ourselves to some special subclasses of oriented matroids, we can see that degree-$k$ oriented matroids, a special class of oriented matroids of rank $k+2$, can be represented by certain geometric configurations in $\mathbb{R}^2$~\cite{M17}. 
However,  as far as the author knows, even such a result is not known for other classes of oriented matroids.
\\
\\
\textbf{Our contribution.}
In this paper, we focus on the class of {\it matroid polytopes}.
Matroid polytopes are a fundamental class of oriented matroids, which provide a precise combinatorial setting for studying polytopes and have played significant roles in polytope theory (see \cite[Chapter 9]{BLSWZ99}).
The class of matroid polytopes lies between polytopes and spheres, but the gaps from those objects are still not well-understood.
To list a few example, shellability of face lattices of matroid polytopes is a famous open problem~\cite[Problem 4.18]{BLSWZ99}.
Existence of a gap between the $f$-vectors of polytopes and those of matroid polytopes is also a well-known open problem~\cite[Problem 4.40]{BLSWZ99}, whereas the corresponding problem for matroid polytopes and spheres
was recently resolved~\cite{B16,BZ18}.

To deepen the understanding of the class of matroid polytopes, it would be useful to have a new type of topological representation theorem for matroid polytopes.
As a first step,  in this paper, we provide a two-dimensional version of the topological representation theorem for uniform matroid polytopes of rank $4$, an abstraction of simplicial $3$-polytopes.
Motivated by the notion of pseudoconfigurations of points~\cite{GP84}, which can represent any rank $3$ oriented matroids, we introduce  {\it $2$-weak configurations of points and pseudocircles ($2$-weak PPC configurations)}, which generalize configurations of points and circles on $S^2$ (Section 3).
It can easily be seen that $2$-weak PPC configurations also determine oriented matroids.
We prove that every uniform matroid polytopes of rank $4$ admits a two-dimensional topological representation as a $2$-weak PPC configuration (Section 4).
Although it is not clear whether this result generalizes into the general-rank case, there is some evidence that a similar result might hold in the general-rank case (see Problem~\ref{prob:general} in Section 6). 

With the new topological representation theorem, we can give a two-dimensional geometric formulation to Las Vergnas simplicial tope conjecture in the case of uniform matroid polytopes of rank $4$.
Based on this formulation, we provide a proof of the conjecture for this class of oriented matroids (Section 5).
\\
\\
\textbf{Notation.}
We will use the following notations in this paper, where $E$ is a set and $(F,<)$ is a totally ordered set.
\begin{itemize}[itemsep=0mm,leftmargin=*]
\item $[a]_F := \{ f \in F \mid f \leq a\}$.
\item $\Lambda^* (F,r) := \{ (i_1,\dots, i_r) \in F^r \mid i_1 < i_2 < \cdots < i_r \}$.
\item $\Lambda (E,r) := \{ (i_1,\dots, i_r) \in E^r \mid i_k \neq i_l \text{ for all } k,l \in [r] \ (k \neq l) \}$.
\item $( i_1,\dots, i_r )^* := (i_{\pi (1)},\dots,i_{\pi (r)})$ for $(i_1,\dots,i_r) \in \Lambda (F,r)$, where $\pi$ is the permutation on $[r]$ such that  $i_{\pi (1)} < \cdots < i_{\pi (r)}$.
\item $X^{\sigma} := \{ e \in E \mid X_e =  \sigma \}$ for $X \in \{ +1,-1,0\}^E$ and $\sigma \in \{ +1, -1, 0\}$.
\item $S(X,Y) := \{ e \in E \mid X_e = - Y_e (\neq 0) \}$ for $X, Y \in \{ +1,-1,0\}^E$.
\end{itemize}

\section{Preliminaries}
We assume that the reader is familiar with oriented matroids.
Here, we only collect basic definitions and facts on oriented matroids.
For details, see \cite{BLSWZ99}.

\subsection{Basic definitions}
There are several equivalent axiom systems for oriented matroids.
In this paper,  we adopt a definition based on the cocircuit axioms.
\begin{defn}(Cocircuit axioms for oriented matroids)\\
Let $E$ be a finite set and ${\cal C}^* \subset \{ +1,-1,0\}^E$ a set of sign vectors.
We call the pair $(E,{\cal C}^*)$  an {\it oriented matroid} if it satisfies the following axioms.
An element of the set ${\cal C}^*$ is called a {\it cocircuit}.
\begin{itemize}
\item[(C0)] $0 \notin {\cal C}^*$.
\item[(C1)] ${\cal C}^* = -{\cal C}^*$.
\item[(C2)] For all $X,Y \in {\cal C}^*$, if $X^0 \subset Y^0$, then $X = Y$ or $X=-Y$.
\item[(C3)] For any $X,Y \in {\cal C}^*$ and any $e \in S(X,Y)$, there exists $Z \in {\cal C}^*$ with $Z^0 = (X^0 \cap Y^0) \cup \{ e \}$, $Z^+ \supset X^+ \cap Y^+$  and  $Z^- \supset X^- \cap Y^-$.
\end{itemize}
\end{defn}
If there exists $k \in \mathbb{N}$ such that $|X^0| = k$ for all $X \in {\cal C}^*$, the oriented matroid $(E,{\cal C}^*)$ is said to be {\it uniform}.
For a set  ${\cal C}^* \subset \{ +1,-1,0\}^E$ satisfying this property, Axiom (C3) can be replaced by the following axiom:
\begin{itemize}
\item[(C3')] For any $X,Y \in {\cal C}^*$ with $|X^0 \setminus Y^0| = 1$ and any $e \in S(X,Y)$, there exists $Z \in {\cal C}^*$ with $Z^0 = (X^0 \cap Y^0) \cup \{ e \}$, $Z^+ \supset X^+ \cap Y^+$  and  $Z^- \supset X^- \cap Y^-$.
\end{itemize}
For a uniform oriented matroid ${\cal M} = (E,{\cal C}^*)$ and $A \subset E$,  we let ${\cal C}|_A^* := \{ X|_A \mid X \in {\cal C}^*, X^0 \subset A \}$ and call the oriented matroid
${\cal M} [A]=(A, {\cal C}|_A^*)$ the {\it restriction} of ${\cal M}$ to $A$.
We let ${\cal C}^* / A := \{ X|_{E \setminus A} \mid X \in {\cal C}^*, X^0 \supset A \}$. Then ${\cal M} / A=(E \setminus A, {\cal C}^* / A)$ is also an oriented matroid, called the {\it contraction} of ${\cal M}$ by $A$.
For $X, Y \in \{ +1,-1,0 \}^E$, we define $X \circ Y \in \{ +1,-1,0\}^E$ by
\begin{align*}
(X \circ Y)_e = 
\begin{cases}
X_e & \text{if $X_e \neq 0$,} \\
Y_e & \text{otherwise}
\end{cases} 
\end{align*}
and let ${\cal V}^*_{\cal M} := \{ X_1 \circ \cdots \circ X_k \mid k \in \mathbb{N}, X_1,\dots,X_k \in {\cal C}^* \} \cup \{ 0 \}$.
An element of ${\cal V}^*_{\cal M}$ is called a {\it covector} of ${\cal M}$.
For $X, Y \in {\cal V}^*_{\cal M}$, we define
\begin{align*}
X \geq Y \Leftrightarrow X_e = Y_e \text{ or } Y_e = 0, \text{ for all } e \in E.
\end{align*}
The rank of ${\cal M} $ is the length $r$ of a maximal chain $0 < X_1 < \cdots < X_r$ of covectors in ${\cal V}^*_{\cal M}$.
If the all-positive vector is contained in ${\cal V}^*_{\cal M}$, the oriented matroid ${\cal M}$ is said to be {\it acyclic}.
The set $F \subset E$ is called a {\it face} of ${\cal M}$ if there is the covector $X_F \in {\cal V}^*_{\cal M}$
with $X^0_F = F$ and $X^+_F = E \setminus F$.
A face $F$ is called an {\it extreme point} (resp. a {\it facet}) if the rank of ${\cal M}[F]$ is $1$ (resp. $r-1$).
If $\{ e \}$ is an extreme point of ${\cal M}$ for every $e \in E$, ${\cal M}$ is said to be a {\it matroid polytope}.
For a matroid polytope ${\cal M}$, a set $F \subset E$ is a face of ${\cal M}$ if and only if ${\cal M}/F$ is acyclic.

Suppose that ${\cal M}$ is a uniform oriented matroid of rank $r$.
For each $\lambda \in \Lambda (E, r-1)$, there are exactly two cocircuits $\pm X_{\lambda} \in {\cal C}^*$ with $X_{\lambda}^0 \supset \lambda$.
We define a map $\chi_{\cal M} : E^r \rightarrow \{ +1, -1, 0\}$ that satisfies the following two conditions.
\begin{itemize}
\item[(B1)] $\chi_{\cal M} (e_{\sigma (1)},\dots,e_{\sigma (r)}) = {\rm sgn}(\sigma )\chi_{\cal M} (e_1,\dots,e_r)$ for all $e_1,\dots,e_r \in E$ and all permutations $\sigma$ on $[r]$.
\item[(B2)] $\chi_{\cal M} (\lambda, e) = X_{\lambda}(e)X_{\lambda}(f)\chi_{\cal M} (\lambda, f)$ for all $\Lambda (E, r-1)$ and all $e,f \in E$.
\end{itemize}
This map is well-defined and is uniquely determined from ${\cal C}^*$ (up to taking the negative).
The map $\chi_{\cal M}$ is called a {\it chirotope} of ${\cal M}$.
It is also possible to specify an oriented matroid by a chirotope; see \cite{BLSWZ99}.

\subsection{Topological representation theorem}
One of the outstanding facts in oriented matroid theory is the Folkman-Lawrence topological representation theorem~\cite{FL78}, which states 
that every oriented matroid rank $d+1$ can be represented by a {\it pseudosphere arrangement} in $S^{d}$.
Here, we briefly review another type of the topological representation theorem established by Goodman and Pollack~\cite{GP84}.
(Our treatment is slightly different from the original in the sense that we consider configurations in the sphere $S^2$ instead of the projective plane $\mathbb{P}^2$.)

\begin{defn}(Pseudoconfigurations of points) \\
Let $E$ be a finite set.
A pair $PP=(P,L)$ of a point set $P:=(p_e)_{e \in E}$ in $S^2$ and a collection $L$ of simple closed curves
is called a {\it pseudoconfiguration of points} (or a {\it generalized configuration of points}) if the following three conditions hold:
\begin{itemize}
\item For any $l \in L$, there exist at least two points of $P$ lying on $l$.
\item For any two points of $P$, there exists a unique curve in $L$ that contains both points.
\item Any pair of distinct curves $l_1,l_2 \in L$ intersects (transversally) exactly twice. 
\end{itemize}
\end{defn}
For each $l \in L$, we label the two connected components of $S^2 \setminus l$  arbitrarily as $l^+$ and $l^-$.
Then, we assign the sign vector $X_{l} \in \{ +1,-1,0\}^E$ 
such that $X_{l}^0 =  \{ e \in E \mid  p_e \in l\}$, $X_{l}^+ = \{ e \in E \mid  p_e \in l^+\}$ and  
$X_{l}^- =\{ e \in E \mid  p_e \in l^-\}$,
and we let ${\cal C}^*_{PP} := \{ \pm X_{l} \mid l \in L\}$.
Then, ${\cal M}_{PP}=(E, {\cal C}^*_{PP})$ is an acyclic oriented matroid of rank $3$.
Goodman and Pollack~\cite{GP84} proved that in fact the converse also holds.
\begin{thm}{\rm (\cite{GP84})}
For any acyclic oriented matroid ${\cal M}$ of rank $3$, there exists a pseudoconfiguration of points $PP$ such that ${\cal M} = {\cal M}_{PP}$.
\end{thm}
Here, the acyclicity condition is not important because it can be removed by considering configurations in $\mathbb{P}^2$ (the original formulation in \cite{GP84}) or signed configurations in $S^2$.
The notion of pseudoconfigurations of points in the $S^d$ can be introduced analogously, but not every acyclic oriented matroid of rank $d+1$
can be represented as a pseudoconfiguration of points in $S^d$. For further details, see \cite[Section 5.3]{BLSWZ99}.

\section{Configurations of points and pseudocircles}
Pseudoconfigurations of points can be viewed as a topological generalization of configurations of points and great circles on the sphere $S^2$.
Here, we introduce a topological generalization of configurations of points and circles (not necessarily great circles) on the sphere $S^2$.
\begin{defn}
Let $E$ be a finite set.
If a point set $P =  (p_e)_{e \in E}$ in $S^2$ and a set $L$ of simple closed curves in $S^2$ satisfy the following conditions, the pair ${\cal P} = (P,L)$ is called a {\it strong configuration of points and pseudocircles} ({\it strong PPC configuration}).
\begin{itemize}
\item[(P1)] For each $l \in L$, there exist at least three points in $P$ lying on $l$.
\item[(P2)] For any three points $p_{e_1}, p_{e_2}, p_{e_3}$ in $P$, there exists a unique curve in $l_{e_1,e_2,e_3}$ that contains all of them.
\item[(P3)] Each pair of distinct curves in $L$ intersects at most twice, where a non-transversal intersection is counted as twice.
\end{itemize}
Let $m \in \{ 1, 2 \}$.
If ${\cal P} = (P,L)$ satisfies (P1), (P2) and the weaker version of (P3) that follows, it is called an {\it $m$-weak configuration of points and pseudocircles} ({\it $m$-weak PPC configuration}).
\begin{itemize}
\item[(P3$^m$)] Each pair of distinct curves in $L$ that shares at least $m$ points in $P$ intersects at most twice, where a non-transversal intersection is counted as twice.
\end{itemize}
\end{defn}
If every curve in $L$ passes through exactly three points in $P$, the configuration ${\cal P}=(P,L)$ is said to be {\it in general position}.
Take $q \in S^2 \setminus \bigcup_{l_{\lambda} \in L}{l_{\lambda}}$ arbitrarily and define,  for each $l_{\lambda} \in L$, $I(l_{\lambda})$ (resp. $O(l_{\lambda})$) to be the connected component of $S^2 \setminus l_{\lambda}$ that contains (resp. does not contain) $q$.
For  each $l_{\lambda} \in L$, we consider a sign vector $X_{\lambda} \in \{ +1,-1,0\}^E$ defined by
\begin{align*}
X_{\lambda}(e) = 
\begin{cases}
+1 & \text{ if $p_e \in O(l_{\lambda}),$} \\
-1 & \text{ if $p_e \in I(l_{\lambda}),$} \\
0 & \text{ if $p_e \in l_{\lambda}$} 
\end{cases}
\end{align*}
and let ${\cal C}^*({\cal P}) := \{ \pm X_{\lambda} \mid \lambda \in \Lambda (E,3) \}$.
Note that ${\cal C}^*({\cal P})$ does not depend on the choice of $q$.
When it is clear from the context what a PPC configuration is considered, we simply write $I_{\lambda}$ and $O_{\lambda}$ to refer to $I(l_{\lambda})$ and $O(l_{\lambda})$.

\begin{defn} 
Let $l_1$ and $l_2$ be simple closed curves in $S^2$.
The closure of each connected component of $S^2 \setminus l_1 \cup l_2$ is called a {\it lens} formed by $l_1$ and $l_2$.
Given a point set $P$, if a lens contains no point of $P$, then the lens is said to be {\it empty} with respect to $P$.
\end{defn}

\begin{prop}
Let $E$ be a finite set with $|E| \geq 4$ and
${\cal P} = ((p_e)_{e \in E}, L)$ a $2$-weak PPC configuration in general position. Then $(E, {\cal C}^*({\cal P}))$ is a uniform matroid polytope of rank $4$.
\label{prop:OM}
\end{prop}
\begin{proof}
We first prove that  ${\cal C}^*({\cal P})$ fulfills the cocircuit axioms.
Clearly, ${\cal C}^*({\cal P}) = -{\cal C}^*({\cal P})$ and $0 \notin {\cal C}^*({\cal P})$ hold.
It is also clear that $X^0 = Y^0$ for $X, Y \in {\cal C}^*({\cal P})$ implies $X=Y$ or $X = -Y$.
Since $\{ X^0 \mid X \in {\cal C}^*({\cal P}) \}$ consists of all $3$-subsets of $E$ (i.e., it determines a uniform matroid of rank $4$),
it suffices to verify the modular cocircuit axiom (C3').
Take $\lambda, \mu \in \Lambda (E,3)$ with $|\bar{\lambda} \cap \bar{\mu} | = 2$ (let us assume that $\lambda_1 = \mu_1$, $\lambda_2 = \mu_2$).
Let $X,Y \in {\cal C}^*({\cal P})$  be ones of the cocircuits that correspond to $l_{\lambda}$ and $l_{\mu}$ respectively.
Remark that $l_{\lambda}$ and $l_{\mu}$ intersect exactly twice and induce four regions $I_{\lambda} \cap I_{\mu}$, $I_{\lambda} \cap O_{\mu}$,
$O_{\lambda} \cap I_{\mu}$, and $O_{\lambda} \cap O_{\mu}$.
By reversing the insides and outsides of $l_{\lambda}$ and $l_{\mu}$ appropriately (i.e., by retaking $q$ appropriately), we can assume $p_{\lambda_3} \in O_{\mu}$, $p_{\mu_3} \in O_{\lambda}$.
Let $e \in S(X, Y)$.
If $p_e \in  I_{\mu} \cap I_{\lambda}$ or $p_e \in O_{\mu} \cap O_{\lambda}$, we have $S(X,Y) = \{ f \in E \mid p_f \in  (O_{\mu} \cap O_{\lambda}) \cup (I_{\mu} \cap I_{\lambda}) \}$  (see Figure~\ref{fig:discussion1a}).
Without loss of generality, let us assume $p_{\lambda_3} \in O_{\lambda_1,\lambda_2,e}$. Then, we have $O_{\lambda_1,\lambda_2,e} \supset I_{\lambda} \cap O_{\mu}$ and $I_{\lambda_1,\lambda_2,e} \supset O_{\lambda} \cap I_{\mu}$.
If we denote by $Z$ the cocircuit that corresponds to $l_{\lambda_1,\lambda_2,e}$ and satisfies $Z(\lambda_3) = Y(\lambda_3)$,
this implies $Z^+ \supset X^+ \cap Y^+$ and $Z^- \supset X^- \cap Y^-$.
If $p_e \in O_{\lambda} \cap I_{\mu}$ or $p_e \in I_{\lambda} \cap O_{\mu}$,
then we have $S(X,Y) = \{ f \in E \mid p_f \in  (O_{\lambda} \cap I_{\mu}) \cup (I_{\lambda} \cap O_{\mu}) \}$ (see Figure~\ref{fig:discussion1b}).
Similarly, we obtain $Z^+ \supset X^+ \cap Y^+$ and $Z^- \supset X^- \cap Y^-$.
Therefore, ${\cal M} = (E,{\cal C}^*({\cal P}))$ is an oriented matroid of rank $4$.

\begin{figure}[h]
\begin{minipage}[t]{0.5\columnwidth}
\centering 
\includegraphics[scale=0.25, bb = 40 240 540 820, clip]{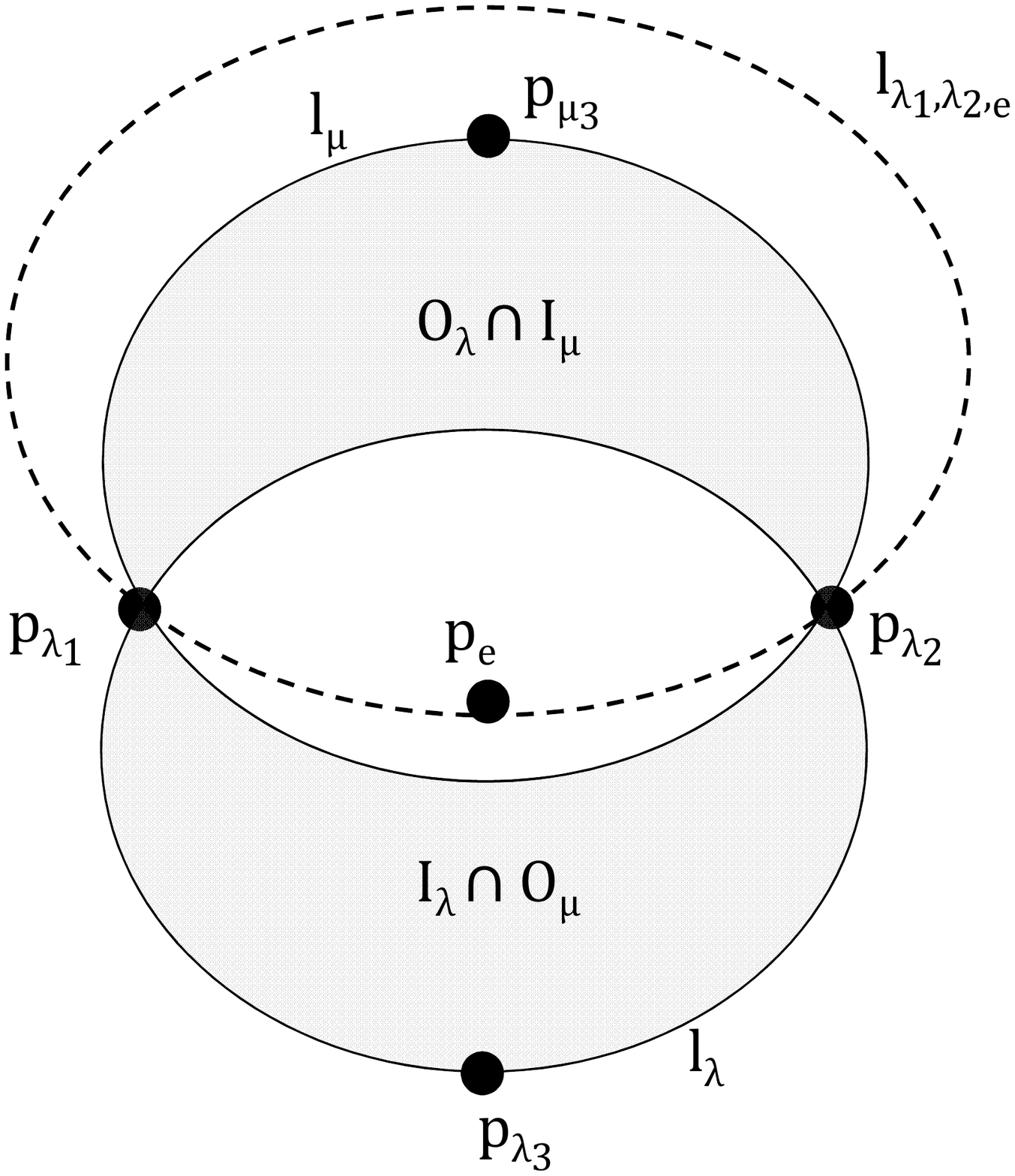} 
\caption{$p_e \in  I_{\mu} \cap I_{\lambda}$}
\label{fig:discussion1a}
\end{minipage}
\begin{minipage}[t]{0.5\columnwidth}
\centering 
\includegraphics[scale=0.25, bb = 50 150 500 720, clip]{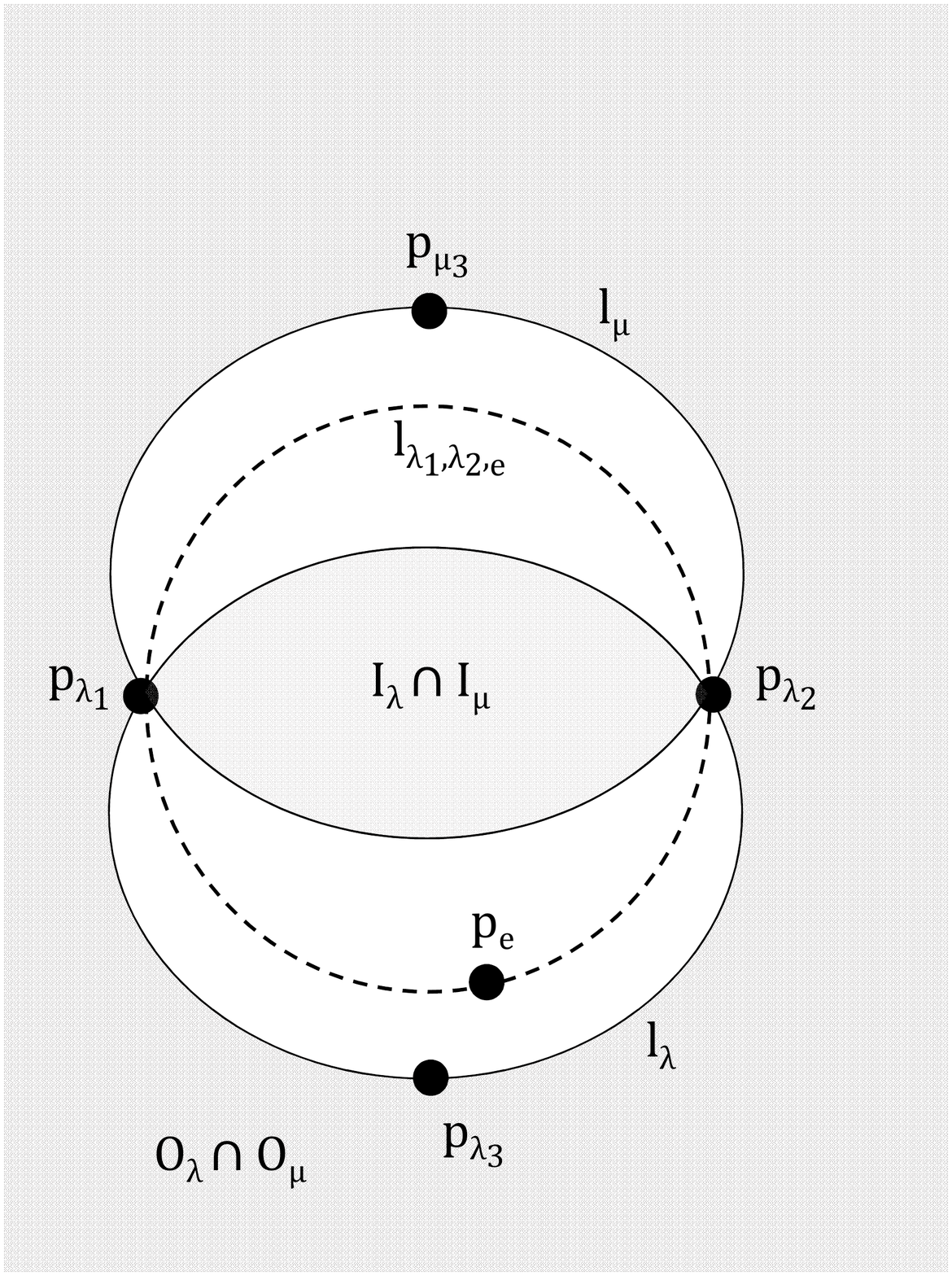} 
\caption{$p_e \in  I_{\mu} \cap O_{\lambda}$}
\label{fig:discussion1b}
\end{minipage}
\end{figure}

Next, assume that ${\cal M}$ is not a matroid polytope.
Then, by the oriented-matroid version of Caratheodory's theorem (see \cite[Chapter 9]{BLSWZ99}), there exists $Q = \{ \lambda_1, \dots, \lambda_5 \} \subset E$ such that ${\cal M}[Q]$ is not a matroid polytope.
Without loss of generality, we assume that $p_{\lambda_4} \in I_{\lambda_1,\lambda_2,\lambda_3}$ and $p_{\lambda_3} \in I_{\lambda_1,\lambda_2,\lambda_4}$.
Then, the curves $l_{\lambda_1,\lambda_2,\lambda_3}$, $l_{\lambda_1,\lambda_2,\lambda_4}$, $l_{\lambda_1,\lambda_3,\lambda_4}$, and $l_{\lambda_2,\lambda_3,\lambda_4}$ induce ten regions (see Figure~\ref{fig:ten_regions}).
\begin{figure}[h]
\centering 
\includegraphics[scale=0.25, bb = 30 206 512 729, clip]{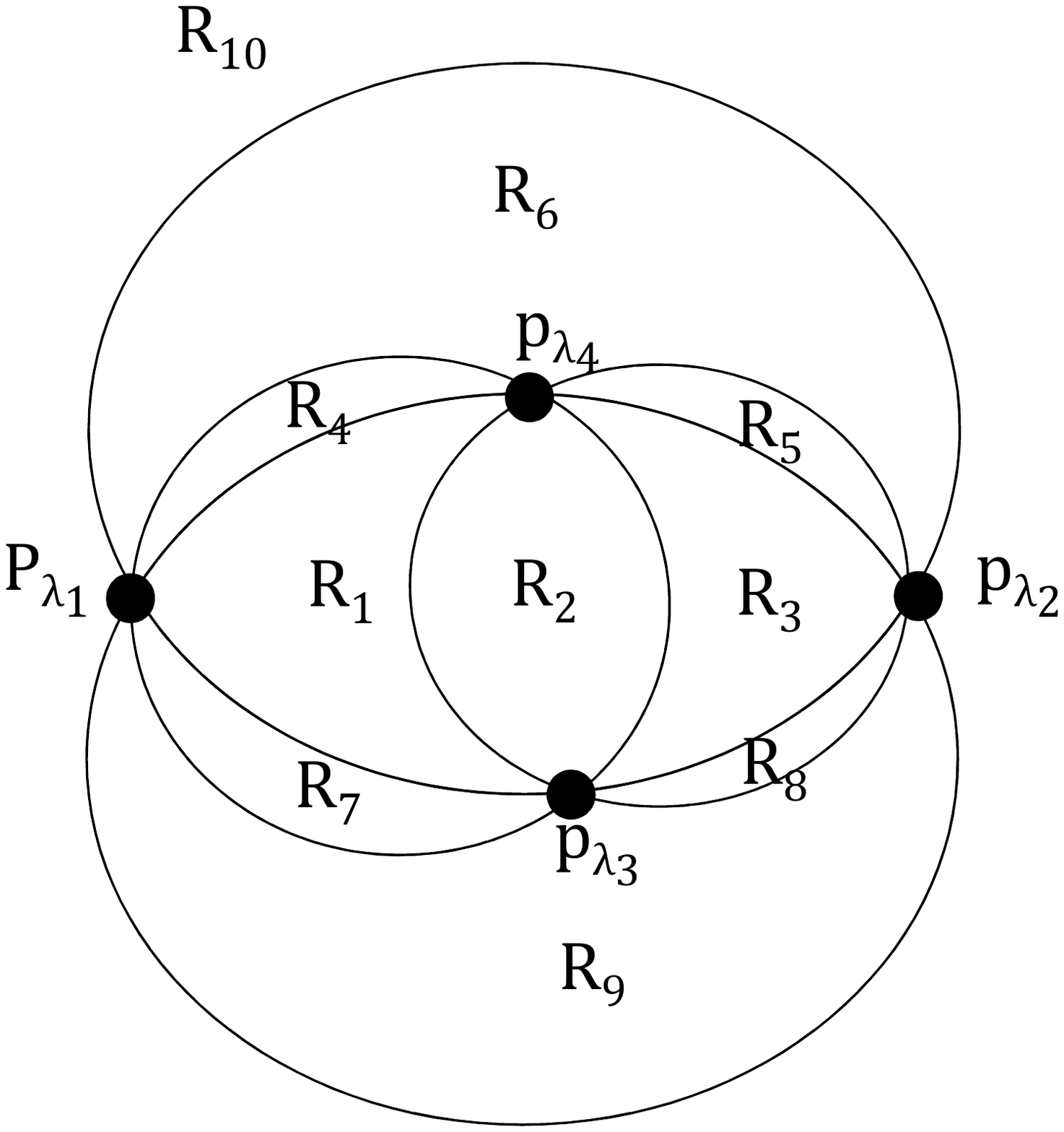} 
\caption{Ten regions induced by the curves $l_{\lambda_1,\lambda_2,\lambda_3}$, $l_{\lambda_1,\lambda_2,\lambda_4}$, $l_{\lambda_1,\lambda_3,\lambda_4}$, and $l_{\lambda_2,\lambda_3,\lambda_4}$}
\label{fig:ten_regions}
\end{figure}
We make a case analysis on the region the point $p_{\lambda_5}$ lies on.
If $p_{\lambda_5} \in I_{\lambda_1,\lambda_2,\lambda_3} \cap  I_{\lambda_1,\lambda_2,\lambda_4} \cap  I_{\lambda_1,\lambda_3,\lambda_4} \cap  I_{\lambda_2,\lambda_3,\lambda_4}$,
then the cocircuits $X_{\lambda_1,\lambda_2,\lambda_3}$, $X_{\lambda_1,\lambda_2,\lambda_4}$, $X_{\lambda_1,\lambda_3,\lambda_5}$, $X_{\lambda_1,\lambda_4,\lambda_5}$, $X_{\lambda_2,\lambda_3,\lambda_5}$, and $X_{\lambda_2,\lambda_4,\lambda_5}$
are non-negative or non-positive.  Without loss of generality, we assume that all of them are non-negative.
Then, the covector $X_{\lambda_1,\lambda_2,\lambda_3} \circ X_{\lambda_1,\lambda_4,\lambda_5}$ indicates that $e_1$ is an extreme point of ${\cal M}[Q]$.
Similarly, $e_2, e_3, e_4, e_5$ are shown to be extreme points.
This leads to that ${\cal M}[Q]$ is a matroid polytope, which is a contradiction.
By applying similar discussions to the other cases, we conclude that ${\cal M}$ must be a matroid polytope.
\end{proof}
\mbox{}
\\
\\
In what follows, we denote by ${\cal M}_{\cal P}$ the oriented matroid of a PPC configuration ${\cal P}$
and by $\chi_{\cal P}$ one of the two chirotopes of ${\cal M}_{\cal P}$.
From the chirotope, we can read off the relative positions of points and pseudocircles:
points $p_i$ and $p_j$ are on the same (resp. opposite) side of  the pseudocircle $l_{\lambda}$ if $\chi_{{\cal P}} (\lambda,i) \cdot \chi_{{\cal P}} (\lambda, j) = +1$ (resp. $\chi_{{\cal P}} (\lambda,i) \cdot \chi_{{\cal P}} (\lambda, j) = -1$).
Conversely, if we want to compute $\chi_{\cal P}$ from the configuration ${\cal P} = ((p_e)_{e \in E}, L)$, we first choose $b \in \Lambda (E, 4)$ arbitrarily and fix the value of $\chi_{\cal P} (b)$.
Then, we decide the other values of  $\chi_{\cal P}$ one by one by considering relative positions of points in $(p_e)_{e \in E}$ and curves in $L$, and by using Axiom (B1).

\begin{exmp}
Let ${\cal P}$ be the strong PPC configuration depicted in Figure~\ref{fig:configuration}.
The oriented matroid ${\cal M_P}$ coincides with the oriented matroid (matroid polytope) arising from the point configuration $Q$ depicted in Figure~\ref{fig:bipyramid}.
Let us assume that $\chi_{\cal P}(1,2,3,4) = +1$. Then, we have $\chi_{\cal P} (1,2,4,3) = -1$ by Axiom (B1).
In the PPC configuration ${\cal P}$, the points $p_3$ and $p_5$ are on the same side of the curve $l_{1,2,4}$, which implies that $\chi_{\cal P}(1,2,4,5) = \chi_{\cal P}(1,2,4,3) = -1$.
Correspondingly, in the point configuration $Q$, the points $q_3$ and $q_5$ are on the same side of the hyperplane $H_{1,2,4}$ spanned by $q_1,q_2,q_4$ (and thus $\triangle{q_1q_2q_4}$ is a facet of the convex hull of $Q$).
On the other hand, the points $p_1$ and $p_5$ are on opposite sides of the curve $l_{2,3,4}$ and thus we have $\chi_{\cal P}(2,3,4,5) = -\chi_{\cal P}(2,3,4,1) = \chi_{\cal P}(1,2,3,4) = +1$.
Correspondingly, the points $q_1$ and $q_5$ are on opposite sides of the hyperplane $H_{2,3,4}$ spanned by $q_2,q_3,q_4$.

\begin{figure}[h]
\begin{minipage}[t]{0.5\columnwidth}
\centering
\includegraphics[scale=0.20, bb = 47 175 590 830, clip]{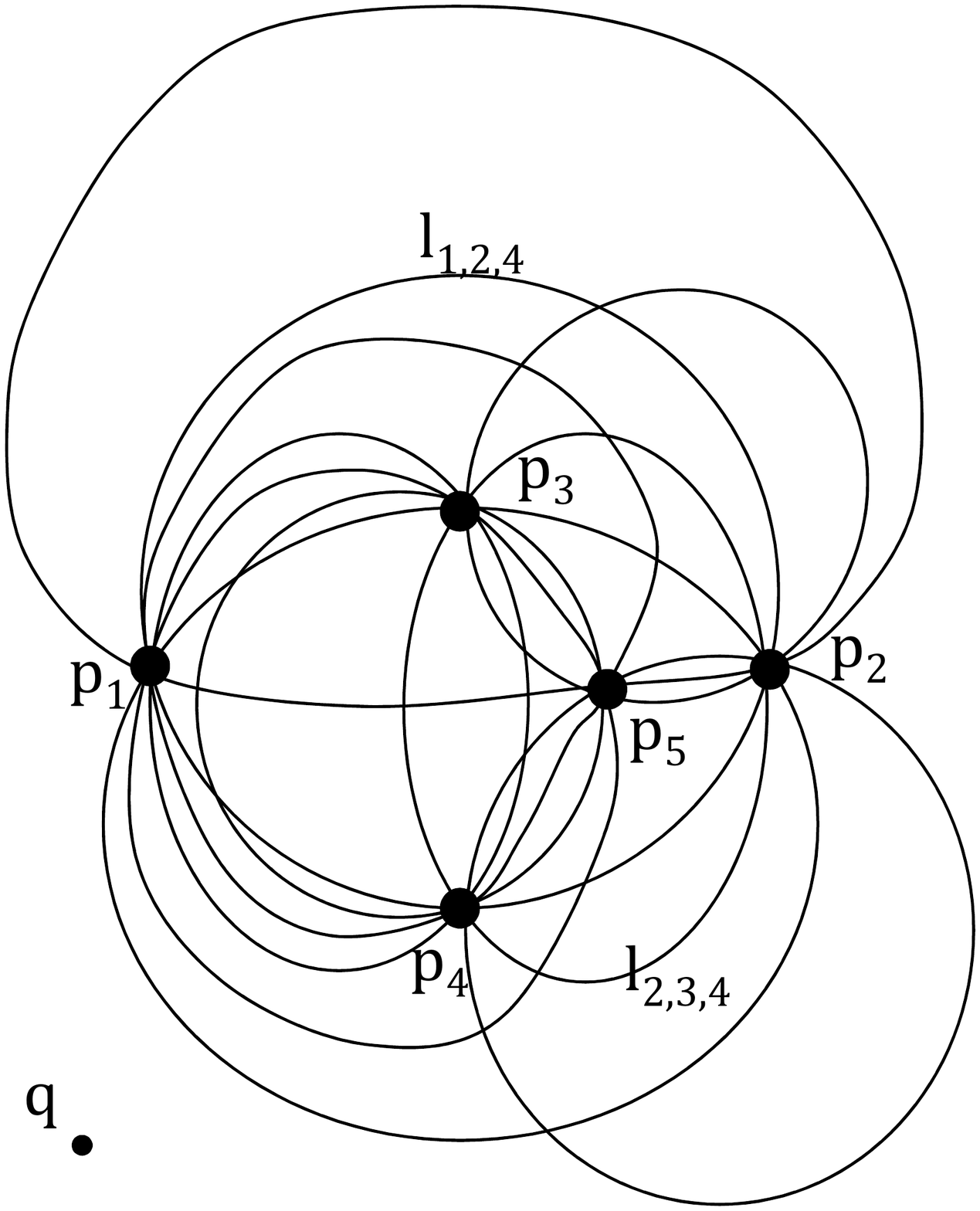} 
\caption{A strong PPC configuration ${\cal P}$}
\label{fig:configuration}
\end{minipage}
\begin{minipage}[t]{0.5\columnwidth}
\centering
\includegraphics[scale=0.20, bb = 60 220 510 810, clip]{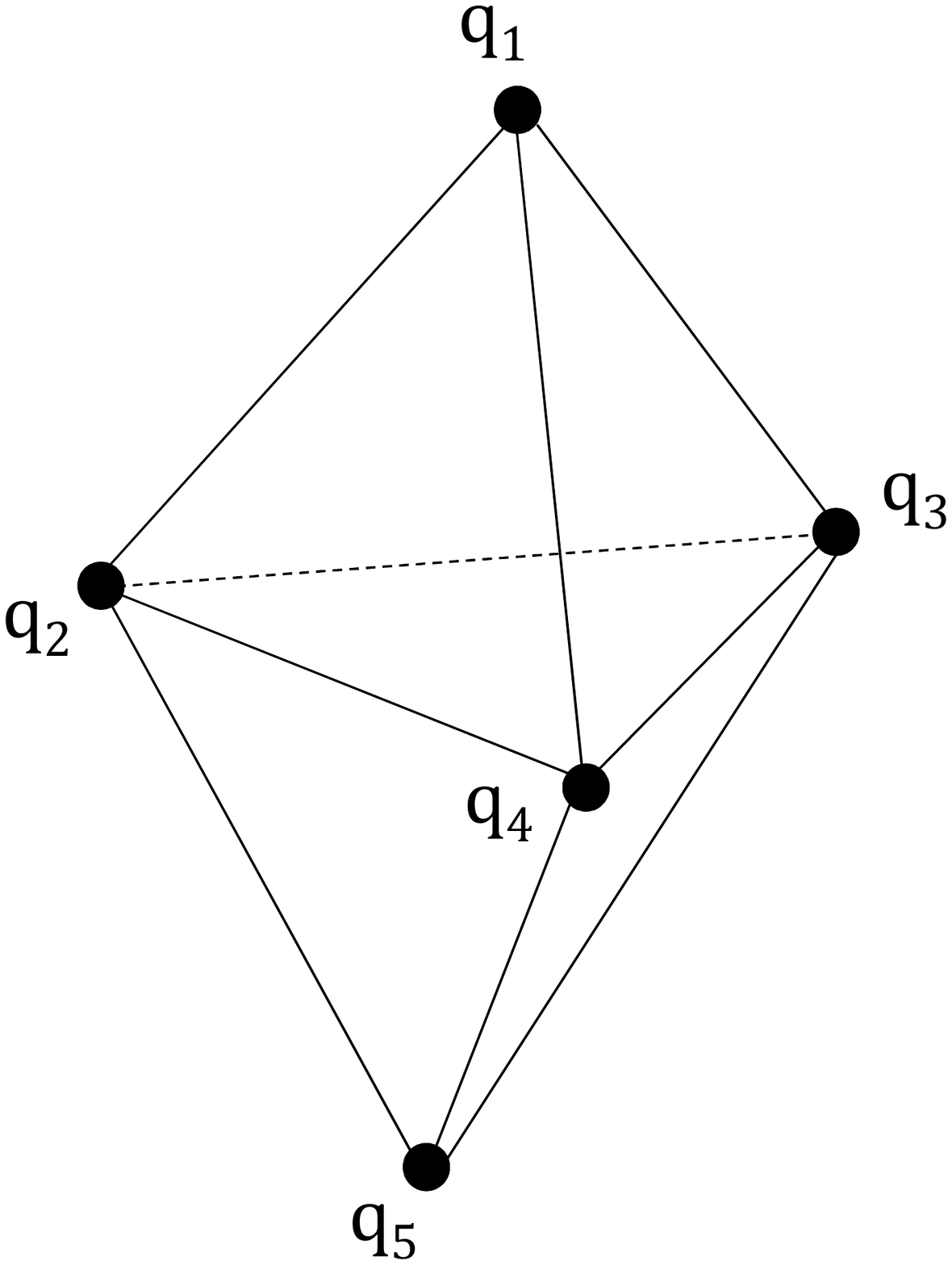} 
\caption{A point configuration $Q$}
\label{fig:bipyramid}
\end{minipage}
\end{figure}

\end{exmp}

\section{Main theorem}
In this section, we prove the main theorem, which is the converse of Proposition \ref{prop:OM}.
\begin{thm}
\label{main_thm}
For every uniform matroid polytope ${\cal M} = (E, {\cal C}^*)$ of rank $4$, there is a $2$-weak PPC configuration ${\cal P}$ in general position with ${\cal C}^* = {\cal C}^*({\cal P})$.
\end{thm}
We will prove this theorem by inductive construction, but the proof is a bit technical.
Thus, we first give a geometric intuition as to why it should be true.
Let $P=(p_e)_{e \in E}$ be a point configuration on the sphere $S^2$ (then, $P$ determines a matroid polytope of rank $4$). 
We assume that $P$ does not contain the north pole of $S^2$ by rotating $P$ slightly if necessary.
Then, for each $p_i,p_j,p_k \in P$ that span a hyperplane, denoted by $H_{i,j,k}$, we have
\begin{align*}
&\text{$p_e$ is above (resp. below, on) $H_{i,j,k}$}\\
& \text{$\Leftrightarrow$ $p_e$ is outside (resp. inside, on) the circle $C_{i,j,k}$ on $S^2$ passing through $p_i, p_j, p_k$.}
\end{align*}
Here, the outside (resp. inside) of $C_{i,j,k}$ is defined to be the side of the circle $C_{i,j,k}$  that contains (resp. does not contain) the north pole of $S^2$. 
Therefore, the oriented matroid arising from a point configuration on  $S^2$ can be represented as a strong PPC configuration.
This can be generalized to realizable matroid polytopes of rank $4$ as follows.
Given a realizable matroid polytope ${\cal M}$ of rank $4$, we take a point configuration $P \subset \mathbb{R}^3$ that realizes ${\cal M}$, and consider
a convex body $K$ that circumscribes the convex hull of $P$, and then apply the $K$-stereographic projection $f_K \colon K \rightarrow \mathbb{R}^2$ (see Figure~\ref{fig:stereograph}).
Then, for each $p_i,p_j,p_k \in P$ that span a hyperplane, denoted by $H_{i,j,k}$, we have
\begin{align*}
&\text{$p_e$ is above (resp. below, on) $H_{i,j,k}$}\\
& \text{$\Leftrightarrow$ $f_K(p_e)$ is outside (resp. inside, on) $f_K(K \cap H_{i,j,k})$ in $\mathbb{R}^2$,}
\end{align*}
where the outside and inside are defined analogously to the case of $S^2$.
By applying the inverse stereographic projection to $f_K(P)$ and $f_K(K \cap H_{i,j,k})$ ($i,j,k \in E$), we obtain a strong PPC configuration that realizes ${\cal M}$.
\\
\begin{figure}[h]
\centering 
\includegraphics[scale=0.25, bb = 10 140 590 560, clip]{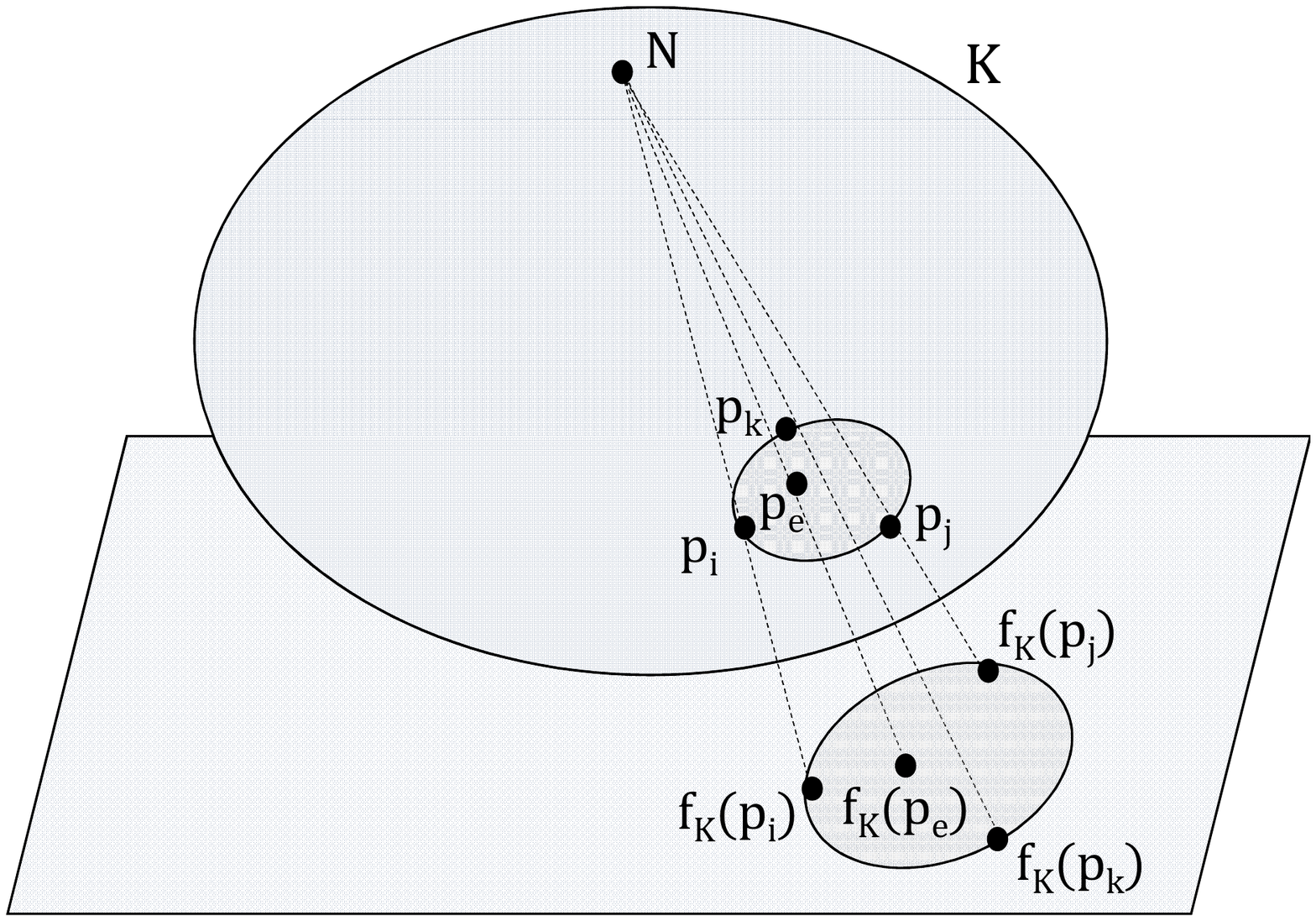} 
\caption{$K$-stereographic projection}
\label{fig:stereograph}
\end{figure}
\\
\noindent
{\bf (Proof of Theorem \ref{main_thm})}

Let $\chi$ be one of the two chirotopes of ${\cal M}$.
We introduce a linear order, denoted by $<$, on $E$ arbitrarily.
Then, we put points $P=(p_{e})_{e \in E}$ in $S^2$ arbitrarily and draw curves $l_{i_1,i_2,i_3}$ one by one in the lexicographic order of  indices $(i_1,i_2,i_3) \in \Lambda^* (E,3)$.
Suppose that curves $l_{j_1,j_2,j_3}$, for $(j_1,j_2,j_3) < (i_1,i_2,i_3)$, are already drawn so that Axiom (P3$^2$) of $2$-weak PPC configurations is satisfied and that the relative positions of the curves with points in $P$ are \emph{consistent} with $\chi$,
i.e., each of the already drawn curves $l_{\lambda}$ separates the points $p_e$ with $\chi (\lambda, e) = +1$ and the points $p_f$ with $\chi (\lambda, f) = -1$.
Then, we insert a curve $l_{i_1,i_2,i_3}$ in the following way.

For each $(i_1,i_2,i_3) \in \Lambda (E,3)$, let $X_{i_1,i_2,i_3}$ be one of the two cocircuits with $X_{i_1,i_2,i_3}^0 = \{ i_1,i_2,i_3 \}$.
For each $(s,t,u) \in \{ (i_1,i_2,i_3), (i_2,i_3,i_1), (i_3,i_1,i_2) \}$, there are at most two cocircuits $X_{s,t,u'} (\neq X_{i_1,i_2,i_3})$ such that
$u' \in [u]_E$ and $S(X_{i_1,i_2,i_3}, X_{s,t,u'}) \cap  [u]_E  = \emptyset$.
If we assume that there are two such cocircuits $X_{s,t,u_+}, X_{s,t,u_-}$, we have $S(X_{s,t,u_+},X_{s,t,u_-}) \cap [u]_E = \{ u \}$.
(In terms of the (partial) 2-weak PPC configuration, we consider $u_+,u_- \in [u]_E$ such that $l_{s,t,u_+},l_{s,t,u_-}$ are the curves ``closest''  to $p_u$ among the already drawn curves passing through $p_s$ and $p_t$.
The curves $l_{s,t,u_+},l_{s,t,u_-}$ induce four connected regions. If we denote by $R$ the region on which $p_u$ lies, we have to connect $p_s$ and $p_u$, and $p_t$ and $p_u$ inside $R$.)
We hereafter assume that $X_{s,t,u_+}$ and $X_{s,t,u_-}$ exist for each $(s,t,u) \in \{ (i_1,i_2,i_3), (i_2,i_3,i_1), (i_3,i_1,i_2) \}$ because otherwise the following discussion becomes simpler.

To construct $l_{i_1,i_2,i_3}$, we first construct arcs connecting $p_{i_1}$ and $p_{i_2}$, connecting $p_{i_2}$ and $p_{i_3}$, and connecting $p_{i_3}$ and $p_{i_1}$, and concatenate them afterwards.
Let $(s,t,u) \in \{ (i_1,i_2,i_3), (i_2,i_3,i_1), (i_3,i_1,i_2) \}$. We give a construction of an arc connecting $p_s$ and $p_t$. 
Since $l_{s,t,u_+}$ and $l_{s,t,u-}$ intersect exactly twice, they induce four connected regions $I_{s,t,u_+} \cap I_{s,t,u_-}$, $I_{s,t,u_+} \cap O_{s,t,u_-}$,
$O_{s,t,u_+} \cap I_{s,t,u_-}$, and $O_{s,t,u_+} \cap O_{s,t,u_-}$.
By reversing the sides of $l_{s,t,u_+}$ and $l_{s,t,u_-}$ appropriately, we assume $p_{u_+} \in I_{s,t,u_-}$ and  $p_{u_-} \in I_{s,t,u_+}$.
Then, we have $p_{u} \in O_{s,t,u_+} \cap O_{s,t,u_-}$ or  $p_u \in I_{s,t,u_+} \cap I_{s,t,u_-}$.
Note that no curve $l_{s,t,k}$ $(k < u)$ intersects $O_{s,t,u_+} \cap O_{s,t,u_-}$ or $I_{s,t,u_+} \cap I_{s,t,u_-}$.
We connect $p_s$ and $p_t$ by an arc $\gamma_{s,t}$ satisfying the following conditions, where $\mathring{\gamma}_{s,t} := \gamma_{s,t} \setminus \{ p_s, p_t \}$.
\begin{multicols}{2}
\begin{minipage}{0.45\textwidth}
Case (I): $p_u \in O_{s,t,u_+} \cap O_{s,t,u_-}$ (see Figure~\ref{fig:discussion2a}).
\begin{itemize}[leftmargin=*]
\setlength{\itemsep}{-3pt}
\item $\mathring{\gamma}_{s,t} \subset I_{s,t,u_+} \cap I_{s,t,u_-}$.
\item  $\gamma_{s,t}$ separates points $p_a$ in $\overline{I_{s,t,u_+}} \cup \overline{I_{s,t,u_-}}$ with $\chi(s,t,u,a) = +1$
from points $p_b$ in $\overline{I_{s,t,u_+}} \cup \overline{I_{s,t,u_-}}$ with $\chi(s,t,u,b) = -1$.
\item  $\mathring{\gamma}_{s,t}$ intersects none of $l_{s,t,a}$, $l_{s,u,b}$, $l_{t,u,c}$ ($a < u, b < t, c < s$).
\end{itemize}
\end{minipage}
\begin{minipage}{0.45\textwidth}
Case (II): $p_u \in I_{s,t,u_+} \cap I_{s,t,u_-}$ (see Figure~\ref{fig:discussion2b}).
\begin{itemize}[leftmargin=*]
\setlength{\itemsep}{-3pt}
\item $\mathring{\gamma}_{s,t} \subset O_{s,t,u_+} \cap O_{s,t,u_-}$.
\item  $\gamma_{s,t}$ separates points $p_a$ in $\overline{O_{s,t,u_+}} \cup \overline{O_{s,t,u_-}}$ with $\chi(s,t,u,a) = +1$
from points $p_b$ in $\overline{O_{s,t,u_+}} \cup \overline{O_{s,t,u_-}}$ with $\chi(s,t,u,b) = -1$.
\item  $\mathring{\gamma}_{s,t}$ intersects none of $l_{s,t,a}$, $l_{s,u,b}$, $l_{t,u,c}$ ($a < u, b < t, c < s$).
\end{itemize}
\end{minipage}
\end{multicols}

\begin{figure}[h]
\begin{minipage}[t]{0.5\columnwidth}
\centering 
\includegraphics[scale=0.25, bb = 20 270 520 810, clip]{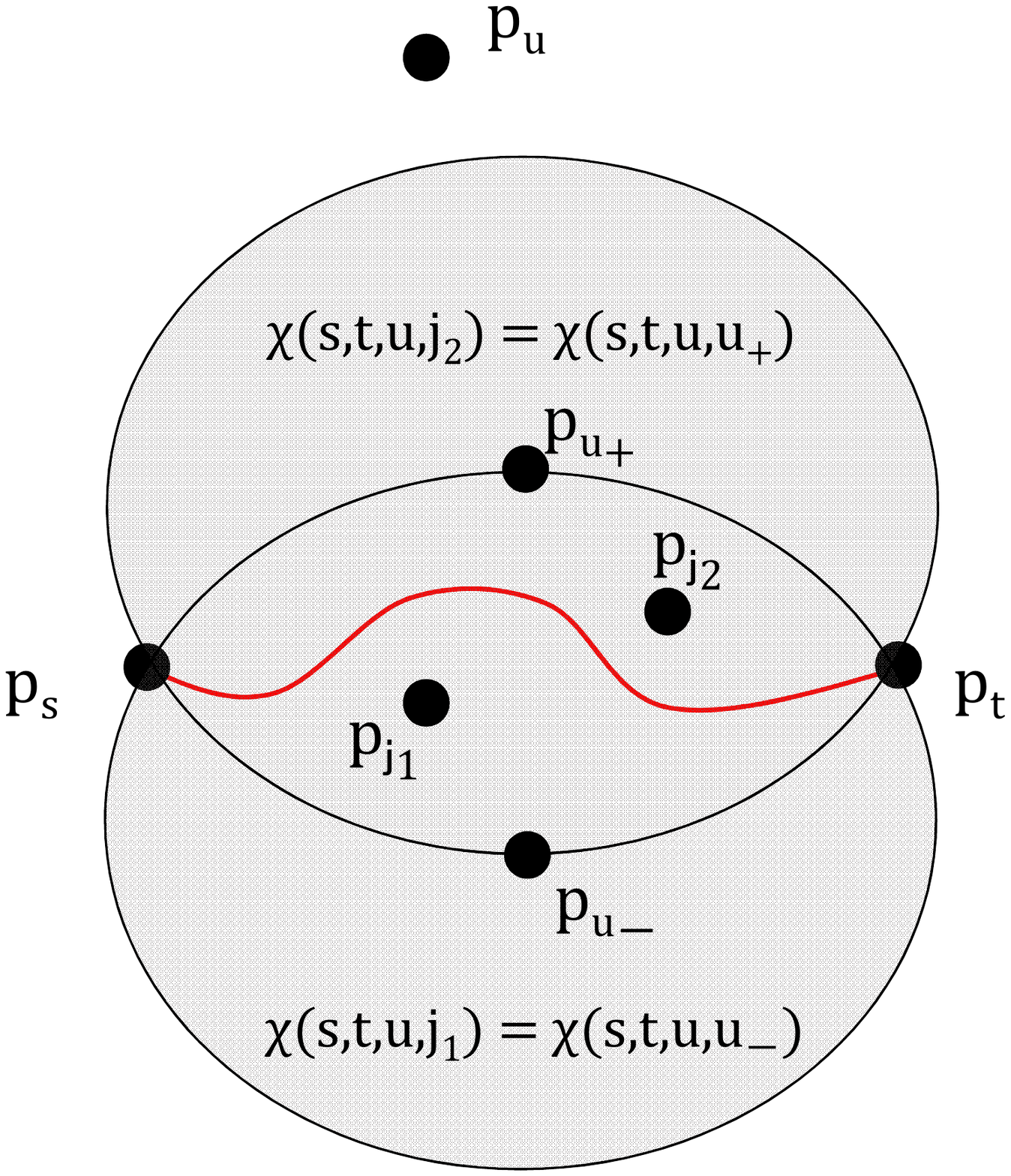} 
\caption{Case (I)}
\label{fig:discussion2a}
\end{minipage}
\begin{minipage}[t]{0.5\columnwidth}
\centering 
\includegraphics[scale=0.25, bb = 20 200 570 820, clip]{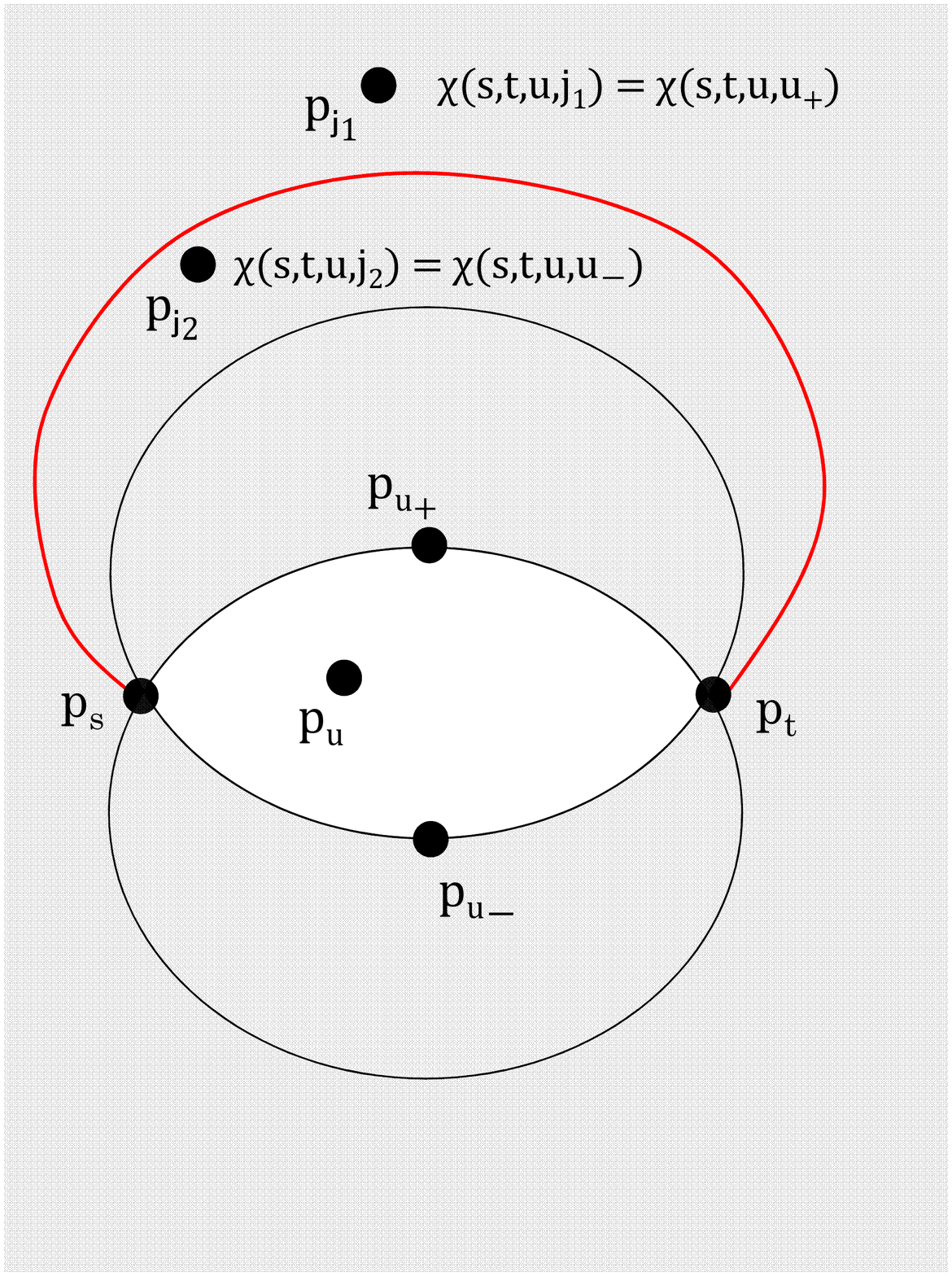} 
\caption{Case (II)}
\label{fig:discussion2b}
\end{minipage}
\end{figure}

\noindent
{\bf Claim 1.}
The curve $\gamma_{s,t}$ can be constructed to satisfy the above conditions.
\\
\\
\noindent
{\bf Proof.}
We assume that a required drawing is impossible.
Then one of the following conditions must hold.
\begin{itemize}
\item[(A-1)] There exist $j_1, j_2 \in E$ such that  $p_{j_1}, p_{j_2} \in \overline{I_{s,t,u_+}} \cap \overline{O_{s,t,u_-}}$  with $\chi(s,t,u,j_1) \cdot \chi(s,t,u,j_2) = -1$. (see Figure~\ref{fig:discussion3_1})
\item[(A-2)] There exist $j_1, j_2 \in E$ such that  $p_{j_1}, p_{j_2} \in \overline{O_{s,t,u_+}} \cap \overline{I_{s,t,u_-}}$ with $\chi(s,t,u,j_1) \cdot \chi(s,t,u,j_2) = -1$. (see Figure~\ref{fig:discussion3_2})
\item[(A-3)] There exist $j_1, j_2 \in E$ such that $p_{j_1} \in \overline{I_{s,t,u_+}} \cap \overline{O_{s,t,u_-}},  p_{j_2} \in \overline{O_{s,t,u_+}} \cap \overline{I_{s,t,u_-}}$ and $\chi(s,t,u,j_1) \cdot \chi(s,t,u,j_2) = +1$.
(see Figure~\ref{fig:discussion3_3})
\item[(B-1)]  For some $l_{s',t',a}$ and $l_{s',u',b}$ ($(s',t',u') \in \{ (i_1,i_2, i_3), (i_2,i_1, i_3), (i_3,i_1,i_2) \}$) with $(s',t',a)^*, (s',u',b)^* < (i_1,i_2,i_3)$ and  $p_{t'} \in O_{s',u', b}$ and $p_{u'} \in I_{s',t',a}$, 
there exist $j_1, j_2 \in E$ such that $p_{j_1}, p_{j_2} \in \overline{O_{s',t',a}} \cap \overline{I_{s',u',b}}$ and $\chi(s',t',u',j_1) \cdot \chi (s',t',u',j_2) = -1$. (see Figure~\ref{fig:discussion3_5})
\item[(B-2)]  For some $l_{s',t',a}$ and $l_{s',u',b}$ ($(s',t',u') \in \{ (i_1,i_2, i_3), (i_2,i_1, i_3), (i_3,i_1,i_2) \}$) with $(s',t',a)^*, (s',u',b)^* < (i_1,i_2,i_3)$ and  $p_{t'} \in I_{s',u', b}$ and $p_{u'} \in O_{s',t',a}$, 
there exist $j_1, j_2 \in E$ such that $p_{j_1}, p_{j_2} \in \overline{I_{s',t', a}} \cap \overline{O_{s',u',b}}$ and $\chi(s',t',u',j_1) \cdot \chi (s',t',u',j_2) = -1$.  (see Figure~\ref{fig:discussion3_6})
\item[(B-3)]  For some $l_{s',t',a}$ and $l_{s',u',b}$ ($(s',t',u') \in \{ (i_1,i_2, i_3), (i_2,i_1, i_3), (i_3,i_1,i_2) \}$) with $(s',t',a)^*, (s',u',b)^* < (i_1,i_2,i_3)$ and  $p_{t'} \in I_{s',u', b}$ and $p_{u'} \in I_{s',t',a}$, 
there exist $j_1, j_2 \in E$ such that $p_{j_1}, p_{j_2} \in \overline{O_{s',t', a}} \cap \overline{O_{s',u',b}}$ and $\chi(s',t',u',j_1) \cdot \chi (s',t',u',j_2) = -1$. (see Figure~\ref{fig:discussion3_7})
\item[(B-4)]  For some $l_{s',t',a}$ and $l_{s',u',b}$ ($(s',t',u') \in \{ (i_1,i_2, i_3), (i_2,i_1, i_3), (i_3,i_1,i_2) \}$) with $(s',t',a)^*, (s',u',b)^* < (i_1,i_2,i_3)$ and  $p_{t'} \in O_{s',u', b}$ and $p_{u'} \in O_{s',t',a}$, 
there exist $j_1, j_2 \in E$ such that $p_{j_1}, p_{j_2} \in \overline{I_{s',t', a}} \cap \overline{I_{s',u',b}}$ and $\chi(s',t',u',j_1) \cdot \chi (s',t',u',j_2) = -1$. (see Figure~\ref{fig:discussion3_8})
\end{itemize}
We first consider Case (A-1).
Let $j \in \{ j_1, j_2 \}$ be such that $\chi (s,t,u,j) = \chi (s,t,u, u_+)$.
Without loss of generality, we assume $\chi (s,t, u_+, j) = +1$.
Since $p_{u_-} \in I_{s,t,u_+}$, we have $\chi (s,t,u_+, u_-) = +1$.
In addition, we have $\chi (s,t,u_+,u) =\sigma$, where $\sigma = +1$ if $p_u \in I_{s,t,u_-} \cap I_{s,t,u_+}$ and  $\sigma = -1$ if $p_u \in O_{s,t,u_-} \cap O_{s,t,u_+}$.
By considering the relative positions of the points $p_{u_+}, p_{j}, p_{u}$ with respect to $l_{s,t,u_-}$, we obtain $\chi (s,t,u_-,j) =  +1$ and $\chi (s,t,u_-,u) =-\sigma$.
By the choice of $j$, we have $\chi (s,t,u,u_+) = \chi (s,t,u,j) = -\sigma$.

From the (partial) information of $\chi$, we compute facets of  ${\cal N} := {\cal M}[s,t,u,u_+,u_-,j]/s$ (recall that $\{ \lambda_1, \lambda_2 \}$ is a facet of ${\cal N}$ 
if and only if the set $\{ \chi (s,\lambda_1, \lambda_2, x) \mid x \in \{ t,u,u_+,u_-,j \} \setminus \{ \lambda_1, \lambda_2 \} \}$ is equal to $\{ +1 \}$ or $\{ -1\}$).
If $\sigma=+1$, then $\{ t, u_+ \}$ is a unique facet of ${\cal N}$ that contains $t$, and  on the other hand  only $\{ t,j\}$ is such a facet if $\sigma=-1$.
However, each element of ${\cal N}$ must be contained in two or zero facets because ${\cal N}$ is a uniform oriented matroid of rank $3$.
(Indeed, consider the contraction ${\cal N}'= {\cal N}/t$.
Since ${\cal N}'$ is of rank $2$, the oriented matroid ${\cal N}'$ is realizable. Since ${\cal N}$ is uniform, the contraction ${\cal N}'$ is simple. Therefore, 
the number of facets of ${\cal N}$ that contain $t$ corresponds to the number of extreme points of ${\cal N}'$, which must be zero or two.)
This is a contradiction.
The same argument applies to Cases (A-2) and (A-3).
\\
\begin{figure}[h]
\begin{minipage}[t]{0.33\columnwidth}
\centering 
\includegraphics[scale=0.25, bb = 40 230 565 833]{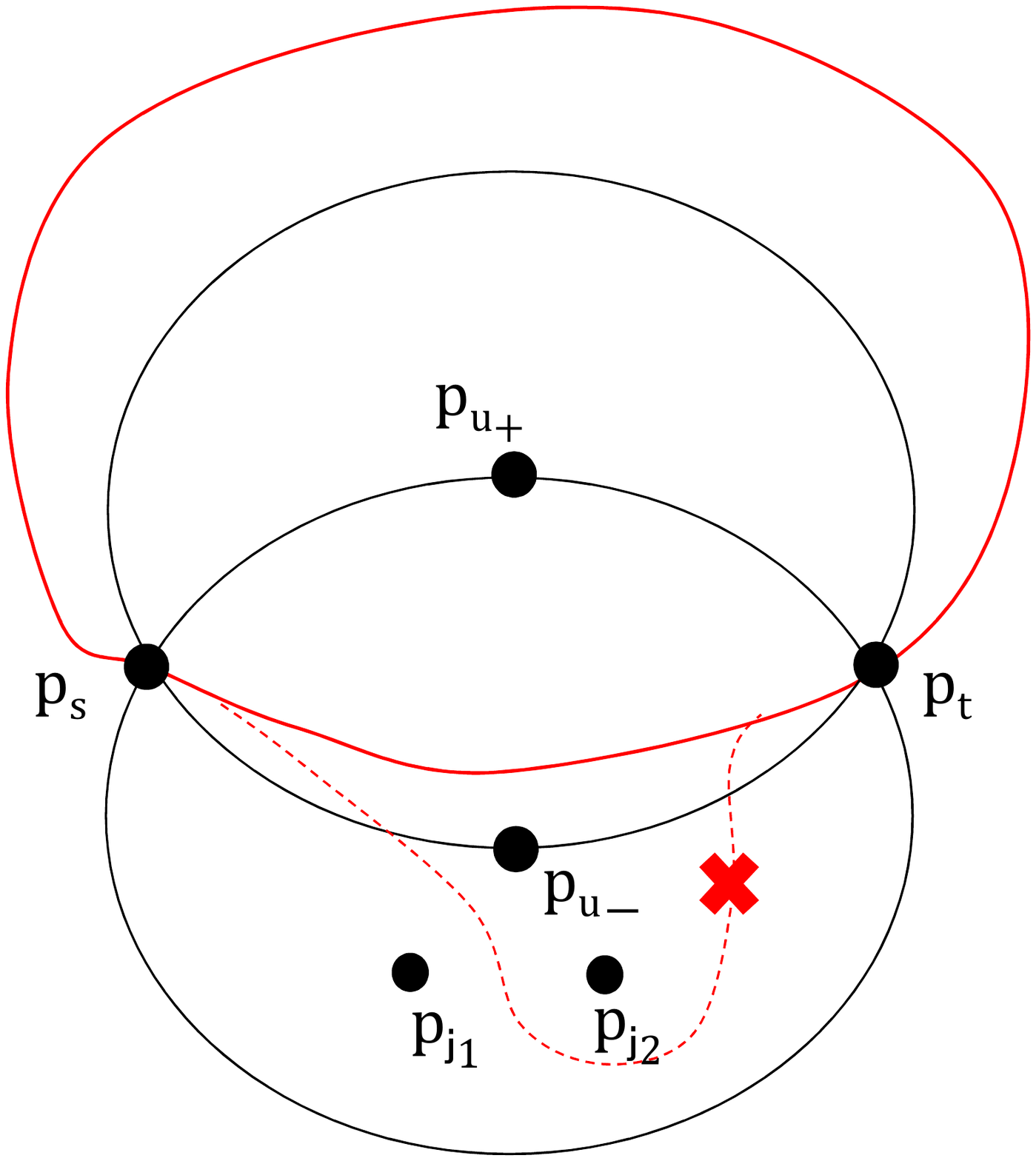} 
\caption{Case (A-1)}
\label{fig:discussion3_1}
\end{minipage}
\begin{minipage}[t]{0.3\columnwidth}
\centering 
\includegraphics[scale=0.25, bb = 40 240 563 842]{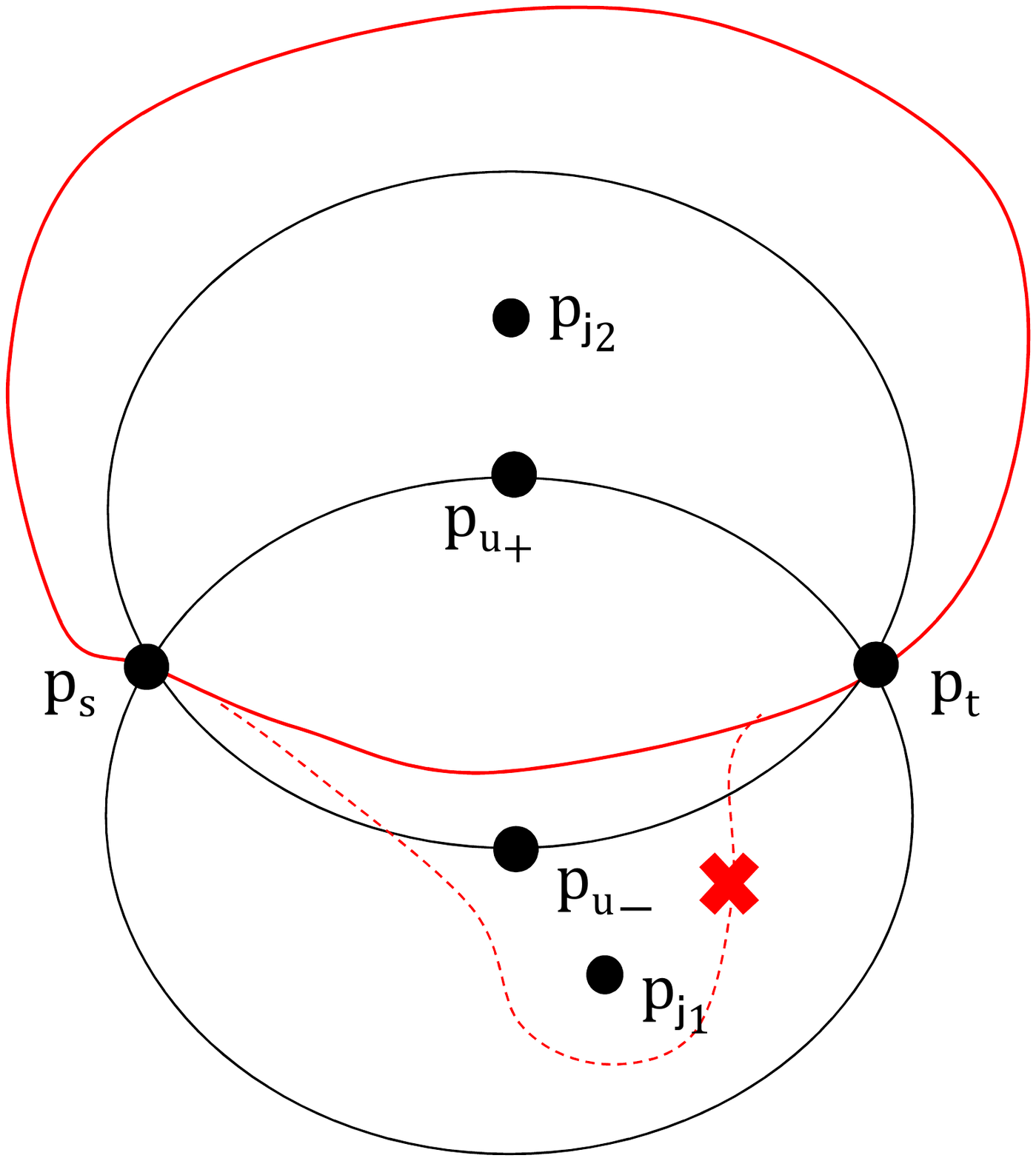}
\caption{Case (A-2)}
\label{fig:discussion3_2}
\end{minipage}
\begin{minipage}[t]{0.33\columnwidth}
\centering 
\includegraphics[scale=0.25, bb = 40 240 563 842]{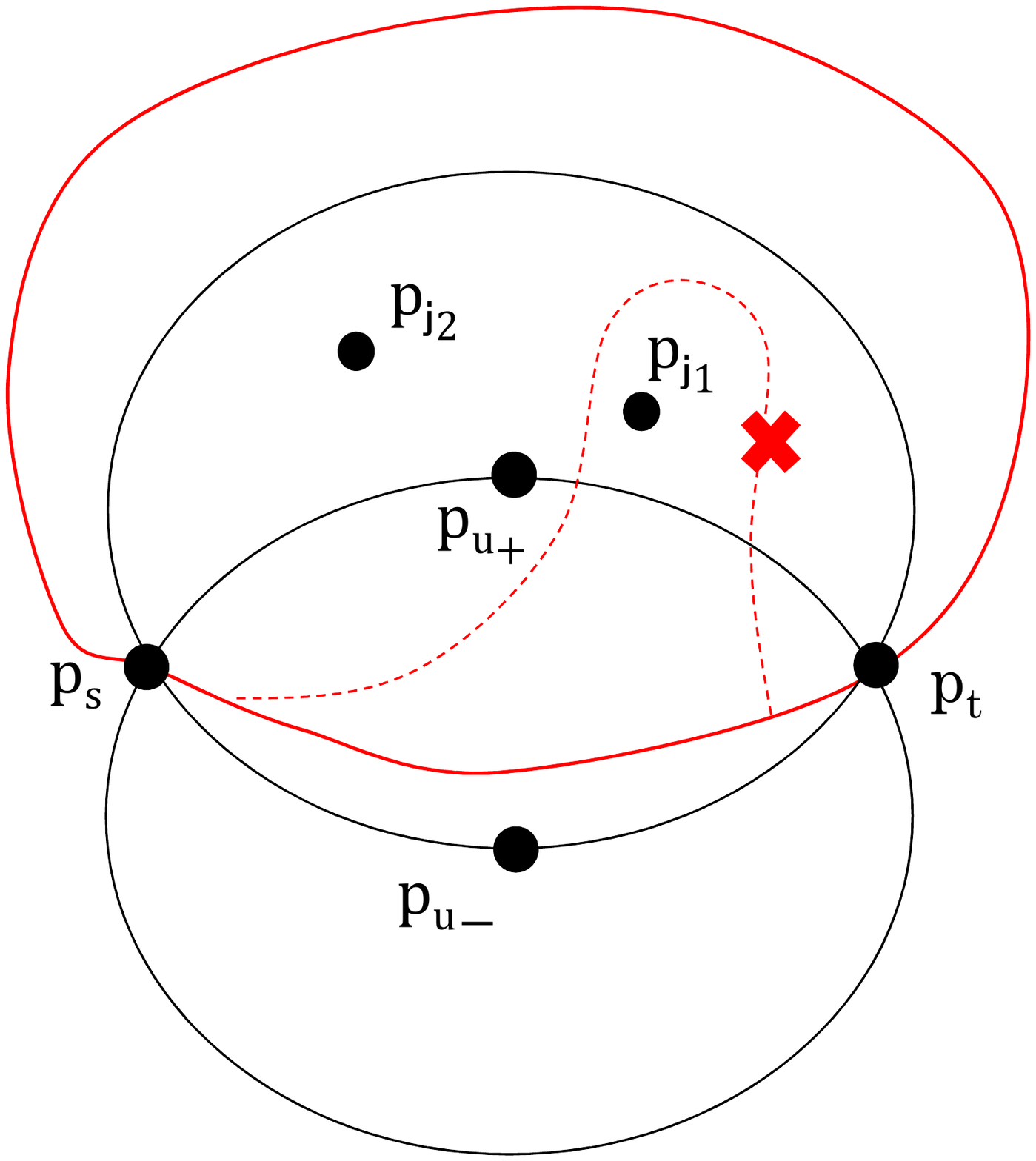}
\caption{Case (A-3)}
\label{fig:discussion3_3}
\end{minipage}
\end{figure}

\noindent
Next, we consider Cases (B1)--(B4).
Since each of Cases (B-1)--(B-4) can be transformed into each other by reversing the insides and outsides of $l_{s',t',a}$ and $l_{s',u',b}$ appropriately, it suffices to consider Case (B-1).
Let us first assume that $(s', t', u') = (i_1,i_2, i_3)$. Then, we have $a < i_3$ and $b < i_2$, which implies that curves $l_{i_1,i_2,a}, l_{i_1,i_3,b}, l_{i_1,i_2,b}$, and $l_{i_1, a,b}$ are already drawn.
\\
\\
(B-1-i) $p_a \in I_{i_1,i_3,b}$ and $p_b \in I_{i_1,i_2,a}$.

Let us first assume that $p_b \notin \wideparen{p_{i_1}p_{i_3}}$.
Let $j \in \{ j_1, j_2 \}$ be such that $\chi (i_1,i_2,i_3,j) = -\chi (i_1,i_2,i_3, b)$.
Since ${\cal M}[i_1,i_2,i_3,a,b,j]$ is a uniform matroid polytope of rank $4$, ${{\cal N} := \cal M}[i_1,i_2,i_3,a,b,j]/i_1$ is an acyclic uniform oriented matroid of rank $3$.
Without loss of generality we assume that $\chi (i_1,i_2,a,b) = +1$. Considering the positions of $p_b,p_{i_3},p_j$ and $l_{i_1,i_2,a}$, we obtain $\chi (i_1,i_2,a,i_3) = +1$ and $\chi (i_1,i_2,a,j) = -1$.  
By the information of $l_{i_1,i_2,b}$, we obtain $\chi (i_1,i_2,b,j) = -1$ and $\chi (i_1,i_2,b,i_3) = +1$.
Considering $l_{i_1,i_3,b}$, we obtain $\chi(i_1,b,i_3,a) = -1$ and $\chi (i_1,b,i_3,,j) = -1$. By the assumption of $j$, we have $\chi (i_1,i_2,i_3,j) = +1$.
Then, none of $\{ i_2, a\}, \{ i_2,b\}, \{ i_2, i_3 \}, \{ i_2, j \}$ is a facet of  ${\cal N}$.
Similarly, none of $\{ b, i_2\}, \{ b, a\}, \{ b, i_3 \}, \{ b, j \}$ is a facet of  ${\cal N}$,
and the same holds for  $\{ i_3, b\}, \{ i_3, i_2 \}$, and $\{ i_3, j \}$.
Since ${\cal N}$ is uniform, there must be two or zero facets of ${\cal N}$ that contains $i_3$.
Therefore, $\{ i_3, a \}$ cannot be a facet of ${\cal N}$ either.
Hence, $i_2$, $i_3$ and $b$ are not extreme points of ${\cal N}$, which contradicts to the fact that an acyclic oriented matroid of rank $r$ has at least $r$ extreme points (see \cite[Theorem 1.3]{LV80}).

Next, we consider the case $b \in \wideparen{p_{i_1}p_{i_3}}$.
Let $j \in \{ j_1, j_2 \}$ be such that $\chi (i_1,i_2,i_3,j) = \chi (i_1,i_2,i_3, b)$ and ${{\cal N} := \cal M}[i_1,i_2,i_3,a,b,j]/i_1$.
In this case, an easy calculation shows that $\{ i_2,b\}$ is a unique facet of ${\cal N}$ that contains $i_2$.
This is a contradiction.
\\
\\
(B-1-ii) $p_a \in O_{i_1,i_3,b}$ and $p_b \in I_{i_1,i_2,a}$.

We consider two cases: (B-1-ii-a) $p_b \in \wideparen{p_{i_1}p_{i_3}}$ and (B-1-ii-b) $p_b \notin \wideparen{p_{i_1}p_{i_3}}$.
In Case (B-1-ii-a), we take $j \in \{ j_1, j_2\}$ such that $\chi (i_1,i_2,i_3,j ) = \chi (i_1,i_2,i_3,b)$. In Case (B-1-ii-b), $j \in \{ j_1, j_2\}$ such that $\chi (i_1,i_2,i_3,j ) = -\chi (i_1,i_2,i_3,b)$.
We then apply a similar discussion.
\\
\\
(B-1-iii) $p_a \in I_{i_1,i_3,b}$ and $p_b \in O_{i_1,i_2,a}$.

Take $j \in \{ j_1, j_2\}$ such that $\chi (i_1,i_2,i_3,j ) = -\chi (i_1,i_2,i_3,b)$ and apply a discussion similar to that of Case~(B-1-i).
\\
\\
(B-1-iv) $p_a \in O_{i_1,i_3,b}$ and $p_b \in O_{i_1,i_2,a}$.

Take $j \in \{ j_1, j_2\}$ such that $\chi (i_1,i_2,i_3,j ) = -\chi (i_1,i_2,i_3,b)$ and apply a discussion similar to Case~(B-1-i).
\\
\begin{figure}[h]
\begin{minipage}[t]{0.5\columnwidth}
\centering 
\includegraphics[scale=0.25, bb = 36 270 556 763, clip]{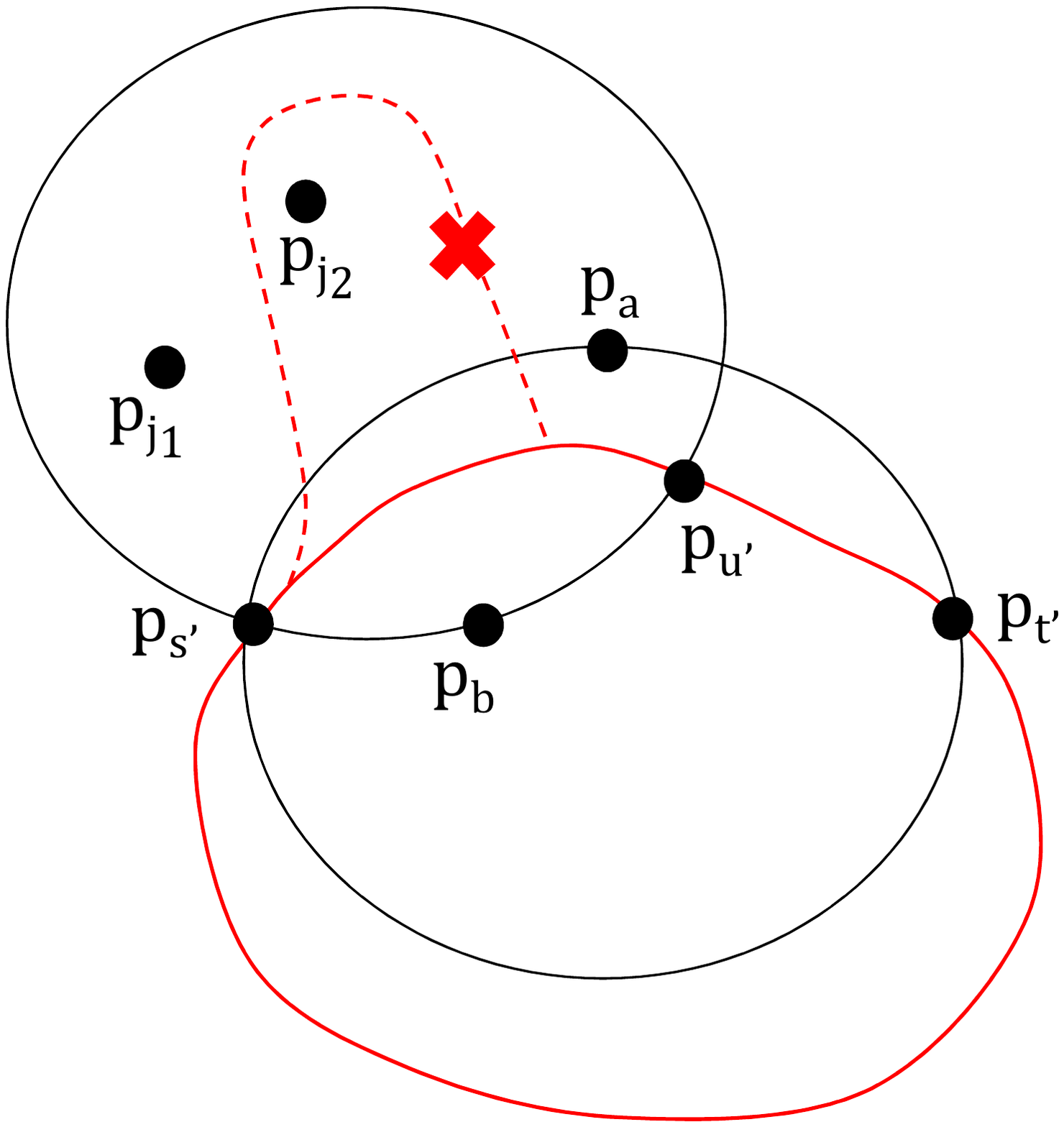}
\caption{Case (B-1)}
\label{fig:discussion3_5}
\end{minipage}
\begin{minipage}[t]{0.5\columnwidth}
\centering 
\includegraphics[scale=0.25, bb = 20 286 540 800, clip]{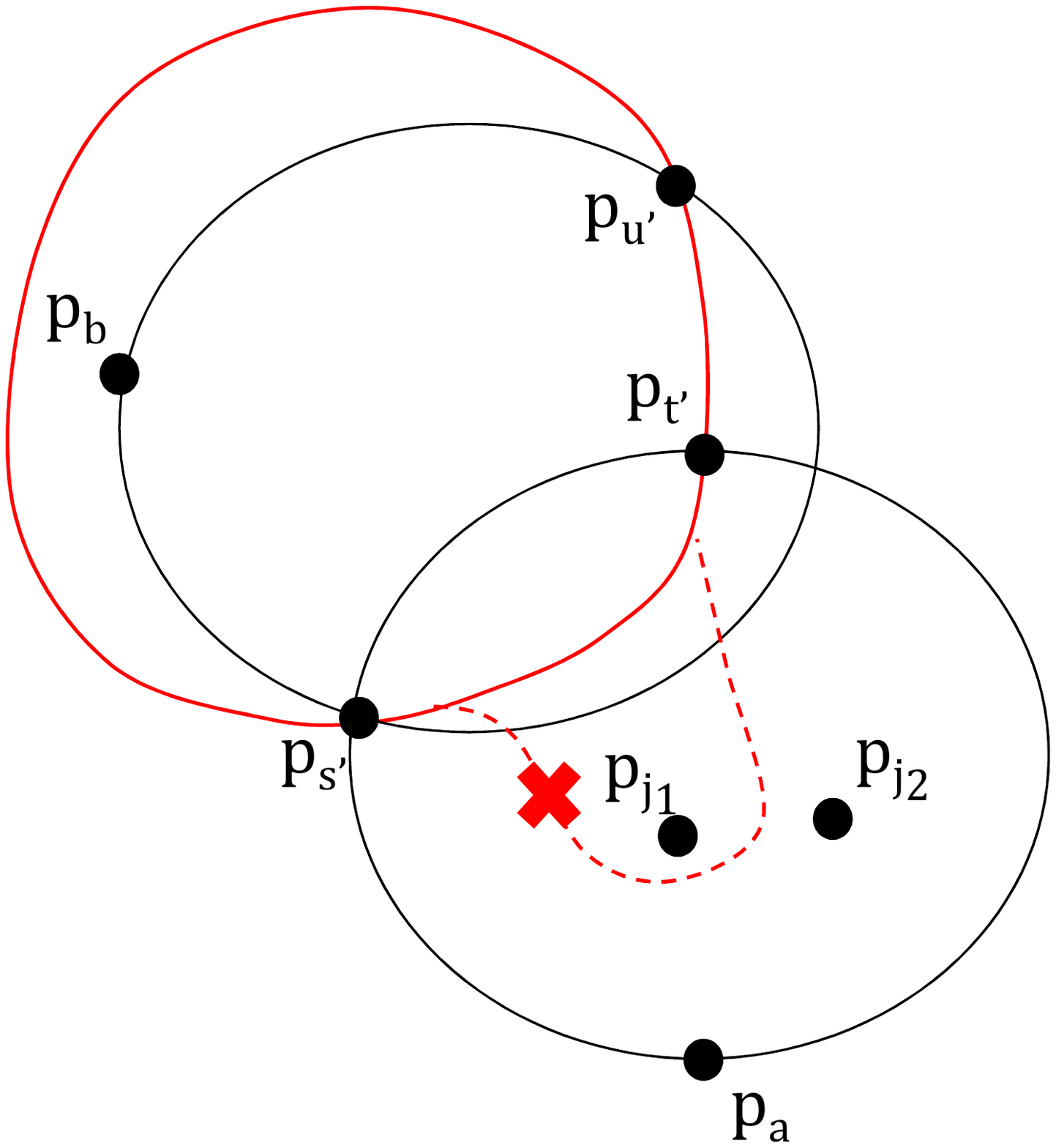}
\caption{Case (B-2)}
\label{fig:discussion3_6}
\end{minipage}
\begin{minipage}[t]{0.5\columnwidth}
\centering 
\includegraphics[scale=0.25, bb = 40 380 540 830, clip]{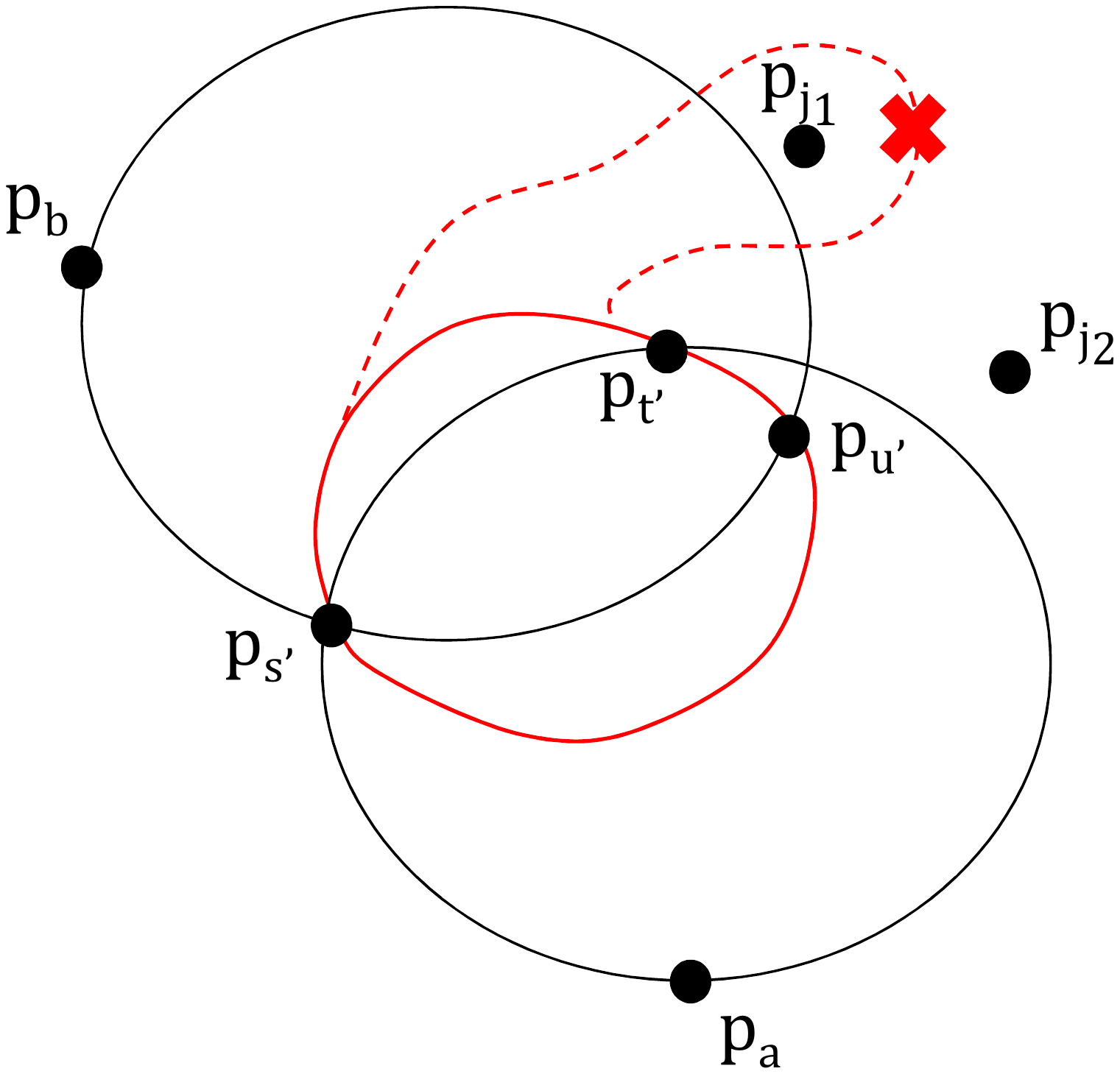}
\caption{Case (B-3)}
\label{fig:discussion3_7}
\end{minipage}
\begin{minipage}[t]{0.5\columnwidth}
\centering 
\includegraphics[scale=0.25, bb = 40 300 540 800, clip]{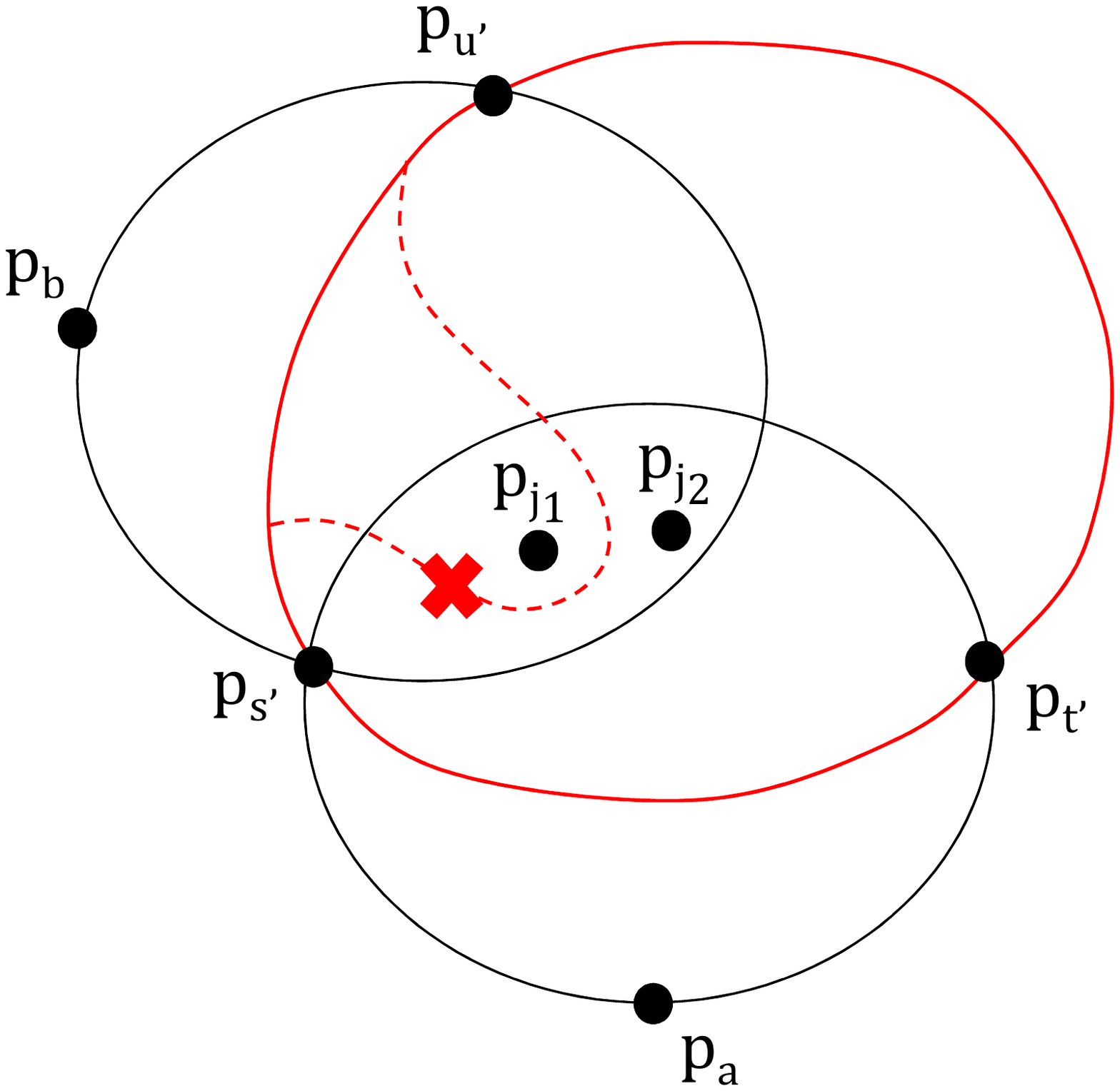}
\caption{Case (B-4)}
\label{fig:discussion3_8}
\end{minipage}
\end{figure}

Therefore, Case (B-1) is impossible if $(s,t,u) = (i_1,i_2,i_3)$. 
The other cases can be treated similarly.
If  $(s,t,u) = (i_2,i_1,i_3)$, we have $a < i_3$ and $b < i_1$, which implies that curves $l_{i_1,i_2,a}, l_{b,i_1,i_2}$ and $l_{b,i_2,i_3}$ are already drawn.
The contraction ${\cal M}/i_2$ must be an acyclic oriented matroid of rank $3$, but we obtain a contradiction similarly to the case $(s,t,u) = (i_1,i_2,i_3)$.
If  $(s,t,u) = (i_3,i_1,i_2)$, we have $a < i_2$ and $b < i_1$, which implies that curves $l_{b,i_1,i_3}, l_{b,a,i_3}, l_{b,i_2,i_3}$, and $l_{i_1,a,i_3,}$ are already drawn.
The contraction ${\cal M}/i_3$ must be an acyclic oriented matroid of rank $3$, but we obtain a contradiction similar to that of the case $(s,t,u) = (i_1,i_2,i_3)$.
\\
\\
Based on the claim, we construct arcs $\gamma_{i_1,i_2}$, $\gamma_{i_2,i_3}$ and $\gamma_{i_3,i_1}$, and
define $l_{i_1,i_2,i_3}$ as their concatenation. Then, the  curve $l_{i_1,i_2,i_3}$ satisfies the desired condition.
\\
\\
{\bf Claim 2.}
The curve $l_{i_1,i_2,i_3}$ is not self-intersecting and it intersects with each of the other drawn curves $l_{i_1,i_2,a}, l_{i_1,i_3,b}, l_{i_2,i_3,c}$ ($a < i_3, b < i_2, c < i_1$) exactly twice.
\\
\\
\noindent
{\bf Proof.}
We first show that $l_{i_1,i_2,i_3}$ intersects with each of the curves $l_{i_1,i_2,a}, l_{i_1,i_3,b}, l_{i_2,i_3,c}$ ($a < i_3, b < i_2, c < i_1$) exactly twice.
It suffices to see that $l_{i_1,i_2,i_3}$ intersects with those curves only transversally at $p_{i_1}, p_{i_2}, p_{i_3}$.
Let us first consider the intersections of $l_{i_1,i_2,i_3}$ and $l_{i_1,i_2,a}$.
If $p_{i_3} \in I_{i_1,i_2,a}$, we have $\mathring{\gamma}_{i_1,i_3} \subset I_{i_1,i_2,a}$ and $\mathring{\gamma}_{i_2,i_3} \subset I_{i_1,i_2,a}$
because neither $\mathring{\gamma}_{i_1,i_3}$ nor $\mathring{\gamma}_{i_2,i_3}$ has an intersection with $l_{i_1,i_2,a}$.
By definition of $\gamma_{i_1,i_2}$,  we have $\mathring{\gamma}_{i_1,i_2} \subset O_{i_1,i_2,a}$.
This means that the curves $l_{i_1,i_2,i_3}$ and $l_{i_1,i_2,a}$ intersect transversally at $p_{i_1}$ and $p_{i_2}$.
The same conclusion is obtained in the case $p_{i_3} \in O_{i_1,i_2,a}$.
The same holds for curves $l_{i_1,i_3,b}$ and $l_{i_2,i_3,c}$.
We can also conclude that $l_{i_1,i_2,i_3}$ is not self-intersecting by applying a similar discussion.
\noindent
\\
\\
Therefore, we can construct $l_{i_1,i_2,i_3}$ with the desired property.
Repeating this procedure, we can realize ${\cal M}$ by a $2$-weak PPC configuration. \QED

\section{Las Vergnas conjecture on simplicial topes}
In \cite{LV80}, Las Vergnas posed the following conjecture:
\begin{conj}
\label{conj:simplicial}
Every simple oriented matroid ${\cal M}$ of rank $r$ has an acyclic reorientation $_{-F}{\cal M}$ that has exactly $r$ extreme points (or equivalently $r$ facets).
\end{conj}
In terms of pseudosphere arrangements, this conjecture claims that every pseudosphere arrangement contains a simplicial cell.
This conjecture has been verified for oriented matroids of rank $3$~\cite{G72}, realizable oriented matroids~\cite{S79} and oriented matroids admitting extensions in general position that are Euclidean~\cite{EM82}, but the original conjecture remains open.
Here, we prove this conjecture for the case of uniform matroid polytopes of rank $4$.

In the following proof, we consider reorientations of uniform matroid polytopes of rank $4$.
Those oriented matroids can be represented as \emph{signed 2-weak PPC configurations}.
Let ${\cal M}=(E,{\cal C}^*)$ be a uniform matroid polytope of rank $4$ and $F$ be a subset of $E$.
To represent the reorientation $_{-F}{\cal M}$, we first take a 2-weak PPC configuration $((p_e)_{e \in E},L)$ that represents ${\cal M}$.
Then, we consider points in $(p_e)_{e \in E \setminus F}$ to be \emph{positive points} and points in $(p_e)_{e \in F}$ to be \emph{negative points}.
For  each curve $l_{\lambda} \in L$, we consider a sign vector $X^F_{\lambda} \in \{ +1,-1,0\}^E$ defined by
\begin{align*}
X^F_{\lambda}(e) = 
\begin{cases}
+1 & \text{ if  $p_e \in O_{\lambda}$ and $e \in E \setminus F$ or if $p_e \in I_{\lambda}$ and $e \in F$,} \\
-1 & \text{ if  $p_e \in I_{\lambda}$ and $e \in E \setminus F$ or if $p_e \in O_{\lambda}$ and $e \in F$,} \\
0 & \text{ if $p_e \in l_{\lambda}$} 
\end{cases}
\end{align*}
(i.e., $X_{\lambda}^F$ is the sign vector obtained by reversing the sign $X_{\lambda}(e)$ of $X_{\lambda}$ for each negative point $p_e$) and we let $({\cal C}^F)^* := \{ \pm X^F_{\lambda} \mid \lambda \in \Lambda (E,3) \}$.
Then, the oriented matroid $(E, ({\cal C}^F)^*)$ coincides with the reorientation $_{-F}{\cal M}$.
In this setting, a set $\{ e_1,e_2,e_3 \} \subset E$ is a facet of $_{-F}{\cal M}$ if and only if the curve $l_{e_1,e_2,e_3}$ separates the positive points and the negative points.

\begin{thm}
Las Vergnas conjecture is true for uniform matroid polytopes of rank $4$.
\end{thm}
\begin{proof}
Let ${\cal M}$ be a uniform matroid polytope of rank $4$. 
Take a  $2$-weak PPC configuration  ${\cal P} = ((p_e)_{e \in E} (=: P), L)$ that realizes ${\cal M}$.
By specifying negative points of ${\cal P}$, we consider reorientations of ${\cal M}$.
For each $e_1,e_2,e_3,e_4 \in E$, we denote $R_{e_1,e_2,e_3,e_4} := (I_{e_1,e_2,e_3} \cap O_{e_1,e_2,e_4}) \cup (O_{e_1,e_2,e_3} \cap I_{e_1,e_2,e_4}) \cup (I_{e_1,e_3,e_4} \cap O_{e_2,e_3,e_4}) \cup (O_{e_1,e_3,e_4} \cap I_{e_2,e_3,e_4})$. 
It suffices to find $p_{e_1^*}, p_{e_2^*}, p_{e_3^*}, p_{e_4^*} \in P$ (and $q \in S^2$) such that $P \cap R_{e_1^*,e_2^*,e_3^*,e_4^*}$ is empty.
Indeed, if there are such $p_{e_1^*}, p_{e_2^*}, p_{e_3^*}, p_{e_4^*} \in P$, we have $p \in (O_{e^*_1,e^*_2,e^*_3} \cap O_{e^*_1,e^*_2,e^*_4}) \cup (I_{e^*_1,e^*_3,e^*_4} \cap I_{e^*_2,e^*_3,e^*_4})$ for all $p \in P \setminus \{ p_{e^*_1}, p_{e^*_2}, p_{e^*_3}, p_{e^*_4} \}$. 
If we let $F := \{ e \in E \mid p_e \in O_{e^*_1,e^*_2,e^*_3} \cap O_{e^*_1,e^*_2,e^*_4} \} \cup \{ e^*_1,e^*_2\}$ and consider the points in $(p_f)_{f \in F}$ to be negative points,
then each of curves $l_{ e^*_1,e^*_2,e^*_3}, l_{ e^*_1,e^*_2,e^*_4}, l_{ e^*_1,e^*_3,e^*_4}, l_{ e^*_2,e^*_3,e^*_4}$ separates the positive points and the negative points.
Thus, the reorientation $_{-F}{\cal M}$ has facets $\{ e^*_1,e^*_2,e^*_3\}$, $\{ e^*_1,e^*_2,e^*_4\}$, $\{ e^*_1,e^*_3,e^*_4\}$, and $\{ e^*_2,e^*_3,e^*_4\}$.
Because an oriented matroid of rank $4$ with such facets cannot have any other facets (see \cite[Lemma 1.4.2]{LV80}),  the oriented matroid $_{-F}{\cal M}$ has exactly four facets.

Points $p_{e_1^*}, p_{e_2^*}, p_{e_3^*}, p_{e_4^*} \in P$  (and $q \in S^2$) are determined in the following manner.
First, we pick up $p_{e_1}, p_{e_2}, p_{e_3}, p_{e_4} \in P$ arbitrarily and
take $q \in S^2$ so that $p_{e_3} \in I_{e_1,e_2,e_4}$ and $p_{e_4} \in I_{e_1,e_2,e_3}$ hold.
If $P \cap R_{e_1,e_2,e_3,e_4} = \emptyset$, we are done.
Suppose that there exists $p_{e_1'} \in P$ such that $p_{e_1'} \in O_{e_2,e_3,e_4} \cap I_{e_1,e_3,e_4}$.
Then, we have  $R_{e'_1, e_2, e_3, e_4} \subset R_{e_1,e_2,e_3,e_4}$.
Since $p_{e'_1} \notin R_{e'_1,e_2,e_3,e_4}$, we have $|P \cap R_{e'_1,e_2,e_3,e_4}| < |P \cap R_{e_1,e_2,e_3,e_4}|$.
Since $p_{e_3} \in I_{e'_1,e_2,e_4}$ and  $p_{e_4} \in I_{e'_1,e_2,e_3}$  hold, we can continue this procedure.
The same discussion can be applied to the case in which there exists $p_{e_2'} \in P$ such that $p_{e_2'} \in I_{e_1,e_3,e_4} \cap O_{e_2,e_3,e_4}$ or
$p_{e_3'} \in P$ such that $p_{e_3'} \in O_{e_1,e_2,e_3} \cap I_{e_1,e_2,e_4}$ or
$p_{e_4'} \in P$ such that $p_{e_4'} \in I_{e_1,e_2,e_3} \cap O_{e_1,e_2,e_4}$.
By repeating this procedure, we obtain $p_{e_1^*}, p_{e_2^*}, p_{e_3^*}, p_{e_4^*}$ such that $R_{e_1^*,e_2^*,e_3^*,e_4^*}$ is empty with respect to $P$.
\begin{figure}[h]
\centering 
\includegraphics[scale=0.25, bb = 30 200 600 700, clip]{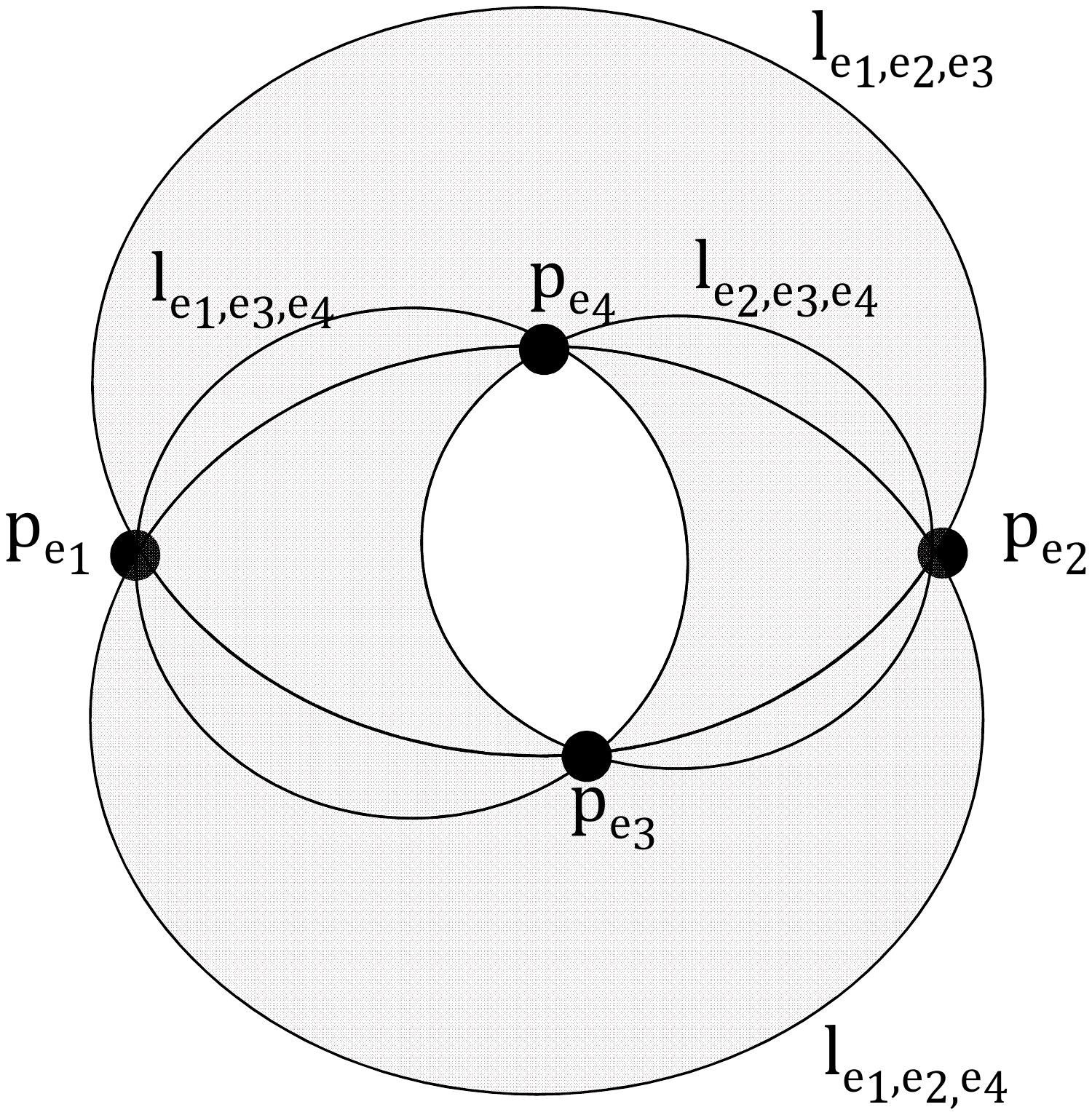} 
\includegraphics[scale=0.25, bb = 30 200 600 700, clip]{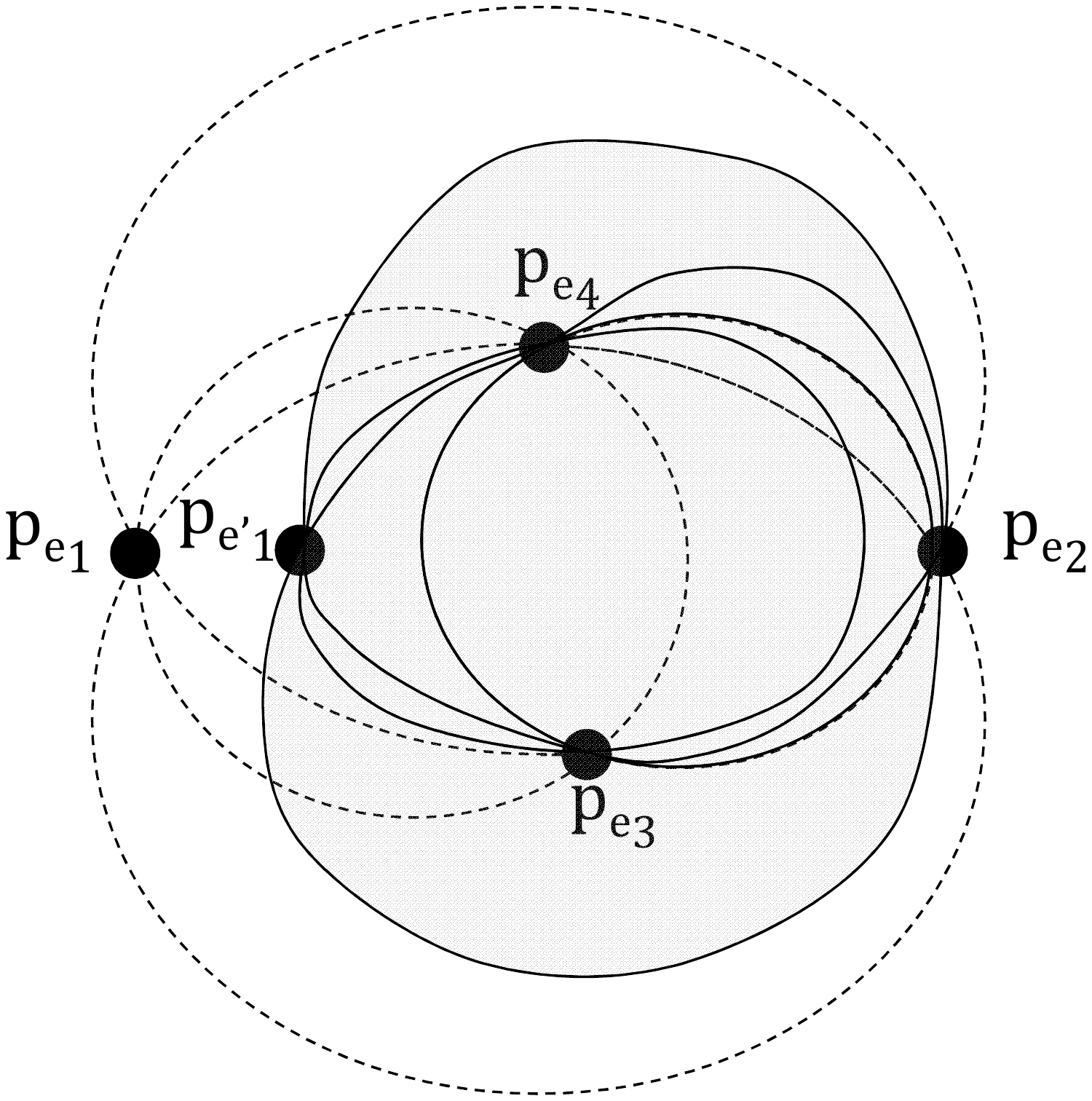}
\includegraphics[scale=0.25, bb = 30 200 600 700, clip]{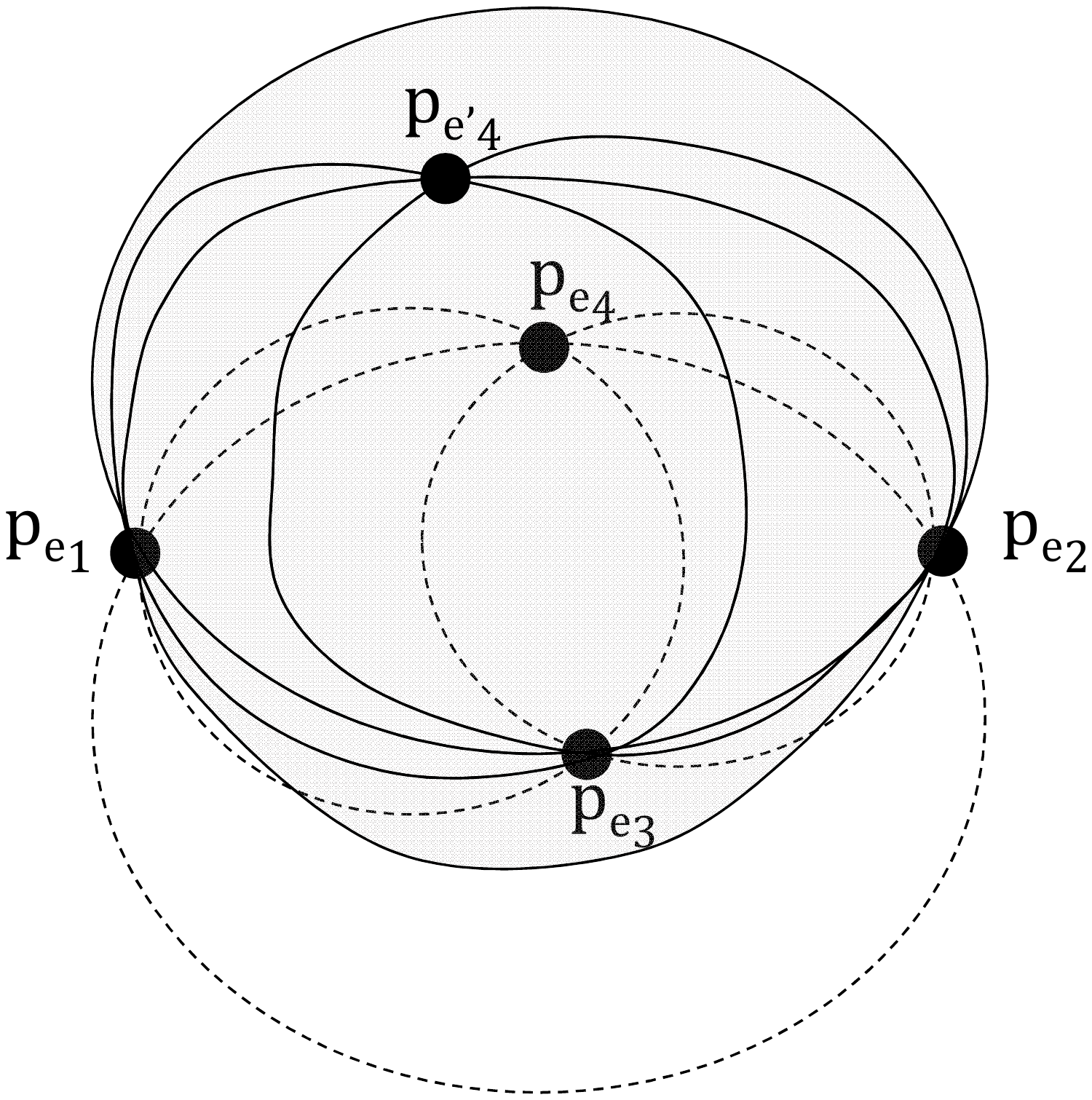}
\caption{$R_{e_1,e_2,e_3,e_4}$, $R_{e'_1,e_2,e_3,e_4}$ and $R_{e_1,e_2,e_3,e'_4}$ (shaded regions). }
\label{fig:discussion8}
\end{figure}
\end{proof}

\begin{remark}
Mandel proved that Las Vergnas conjecture is true for oriented matroids that have extensions in general position that are Euclidean~\cite{EM82}.
As a ``wishful thinking statement'', Mandel~\cite{EM82} provided a conjecture that every oriented matroid has such an extension.
If this conjecture is proved, Las Vergnas conjecture is immediately proved for all oriented matroids.
However, this conjecture is wide open, even for uniform matroid polytopes of rank $4$.
\end{remark}

\section{Concluding remarks}
In this paper, we proved that every uniform matroid polytopes of rank $4$ can be represented by a $2$-weak PPC configuration in general position.
We focused on the uniform case, but we believe that it is not difficult to extend the result to the non-uniform case.
There remain many open problems.
\begin{prob}
Can every uniform matroid polytope of rank 4 be represented by a strong PPC configuration or a $1$-weak PPC configuration?
\end{prob}
After realizing a matroid polytope by a $2$-weak PPC configuration, we can eliminate some redundant intersections (see Figures~\ref{fig:discussion5} and \ref{fig:discussion6}).
However, these operations cannot be applied if $R_1$, $R_2$, $R_3$, and $R_4$ are all non-empty. 
Is it possible to modify the proof of Theorem~\ref{main_thm} to represent the matroid polytopes by strong PPC configurations or $1$-weak PPC configurations?
\begin{figure}[h]
\centering 
\includegraphics[scale=0.4, bb = 28 522 563 842,clip]{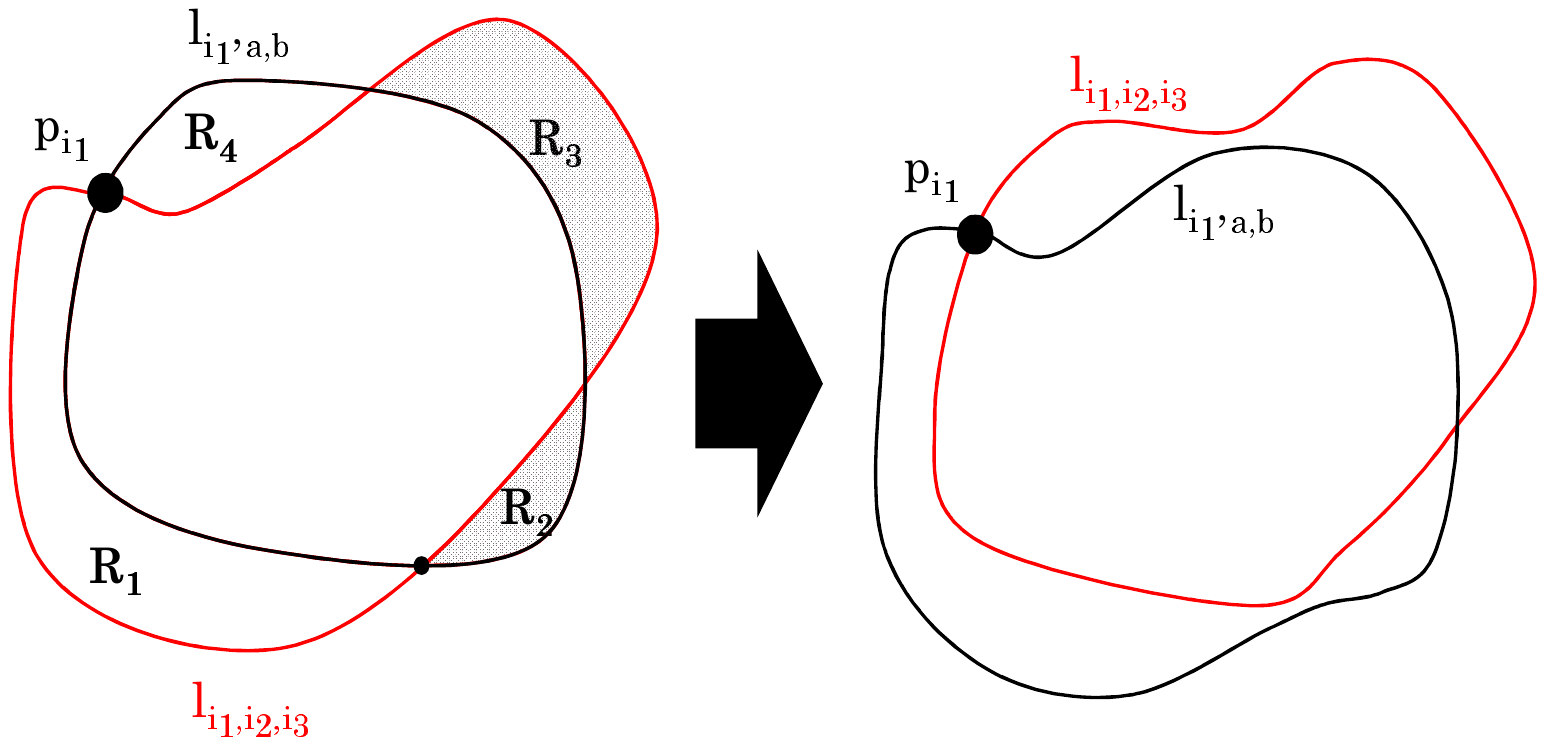}
\caption{Eliminating unnecessary intersection points ($R_3$ is empty)}
\label{fig:discussion5}
\includegraphics[scale=0.4, bb = 32 550 500 800,clip]{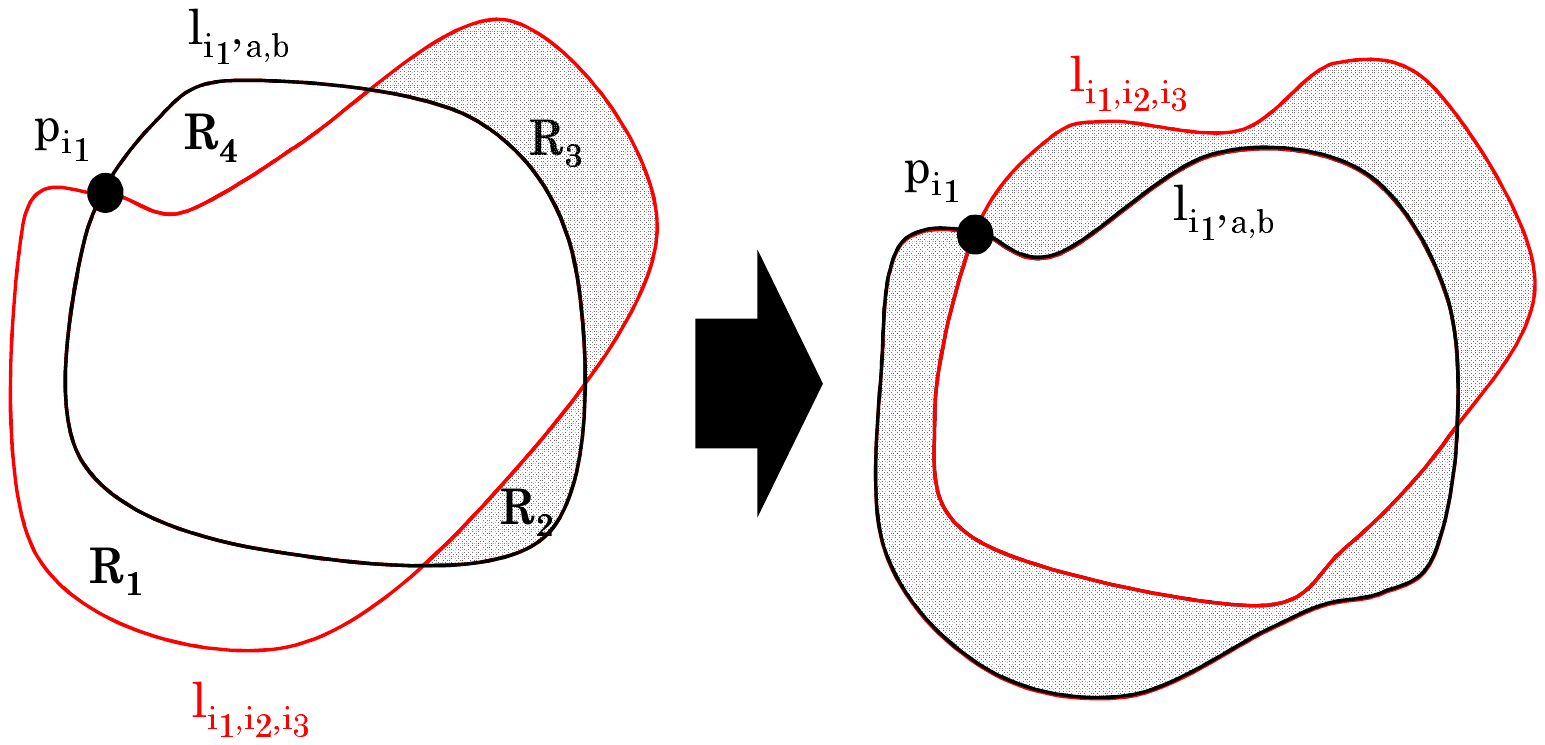}
\caption{Eliminating unnecessary intersection points ($R_1 \setminus \{ p_{i_1} \}$ and $R_4  \setminus \{ p_{i_1} \}$ are empty)}
\label{fig:discussion6}
\end{figure}

\begin{prob}
Is there a two-dimensional topological representation theorem for general oriented matroids of rank $4$?
\end{prob}
Reorientations of uniform matroid polytopes of rank $4$ can be represented by signed PPC configurations.  
However, there are oriented matroids that cannot be reoriented into matroid polytopes.
Thus, it would be natural to consider whether general oriented matroids of rank $4$ can be represented by two-dimensional topological objects. 
\begin{prob}
Can every matroid polytope be represented by a configuration of points and pseudospheres (under a suitable definition)?
\label{prob:general}
\end{prob}
In Section 4, we explained how to represent a realizable matroid polytope of rank $4$ as a PPC configuration via stereographic projection.
That approach can easily be generalized to the general-rank case. 
Thus, one promising method might be to generalize that approach to general matroid polytopes.
\begin{prob}
Can every pair of uniform matroid polytopes of rank $4$ on the same ground set be transformed into each other by a finite number of mutations, preserving matroid polytope property?
\end{prob}
Of course, this problem is motivated by Cordovil-Las Vergnas conjecture:
\begin{conj}(Cordovil-Las Vergnas)
Every pair of uniform oriented matroids of rank $r$ on the same ground set can be transformed into each other by a finite number of mutations. 
\end{conj}
This conjecture is verified only for rank $3$ oriented matroids~\cite{R56} and realizable oriented matroids~\cite{RS88}.
The above problem asks if matroid polytopes can be transformed, while keeping the condition of matroid polytopes.
The general rank version of this problem is also discussed in \cite{RS88}.

\section*{Acknowledgement}
This research was supported by JSPS KAKENHI Grant Numbers JP26730002, JP19K2021000.


\begin{thebibliography}{99}
\bibitem{BL78} R. Bland and M. Las Vergnas, Orientability of matroids, {\it J. Combin. Theory Ser. B} 24(1):94--123, 1978.
\bibitem{BLSWZ99} A. Bj\"orner, M. Las Vergnas B. Sturmfels, N. White and G.M. Ziegler, {\it Oriented Matroids (2nd edition)},
Cambridge University Press, 1999.
\bibitem{B16} P. Brinkmann, {\it $f$-Vector Spaces of Polytopes, Spheres, and Eulerian Lattices}, Ph. D. Thesis, Freie Universit\"{a}t Berlin, 2016.
\bibitem{BZ18} P. Brinkmann and G.M. Ziegler, Small $f$-vectors of $3$-spheres and of $4$-polytopes,  {\it Math. Comp.} 87:2955--2975, 2018. 
%
%
\bibitem{EM82} J. Edmonds and A. Mandel, {\it Topology of oriented matroids}, Ph. D. Thesis, University of Waterloo, 1982.
\bibitem{FL78} J. Folkman and J. Lawrence, Oriented Matroids, {\it J. Combin. Theory Ser. B}, 25(2):199--236, 1978.
\bibitem{GP84} J.E. Goodman and R. Pollack, Semispaces of configurations, cell complexes of arrangements, {\it J. Combin. Theory Ser. A}, 37(3):257--293, 1984.
\bibitem{G72} B. Gr\"unbaum, {\it Arrangements and Spreads}, Regional Conf. Series in Math. 10, Amer. Math. Soc., 1972.
\bibitem{LV80} M. Las Vergnas, Convexity in oriented matroids, {\it J. Combin. Theory Ser. B}, 29(2):231--243, 1980.
\bibitem{M17} H. Miyata, On combinatorial properties of points and polynomial curves, arXiv:1703.04963, 14 pages, 2017.
\bibitem{R56}  G. Ringel, Teilungen der Ebene durch Geraden oder topologische Geraden, {\it Math. Z}., 64:79--102, 1956.
\bibitem{RS88} J.-P. Roudneff and B. Sturmfels, Simplicial cells in arrangements and mutations of oriented matroids,  {\it Geom. Dedicata}, 27(2):153--170, 1988. 
\bibitem{S79} R.W. Shannon, Simplicial cells in arrangements of hyperplanes, {\it Geom. Dedicata} 8(2):179--187, 1979.
\bibitem{Z95} G.M. Ziegler, {\it Lectures on Polytopes}, Springer-Verlag, Berlin, 1995.
\end{thebibliography}
\end{document}